\documentclass[english, a4paper ,DIV=12, parskip=full, twoside]{scrarticle}
\usepackage{local-preamble}
\usepackage{biblatex}
\shortauthor{Bergmann, R., Ferreira, O. P., Santos, E. M., Souza, J. C. O.}
\shorttitle{The difference of convex algorithm on Hadamard manifolds}
\date{May 2nd, 2023}
\author{%
    Ronny Bergmann%
    \thanks{%
        Norwegian University of Science and Technology, Department of Mathematical Sciences, Trondheim, Norway (\email{ronny.bergmannn@ntnu.no}, \url{https://www.ntnu.edu/employees/ronny.bergmann}, \orcid{0000-0001-8342-7218}).%
    }
    \and
    Orizon P. Ferreira%
    \thanks{%
        IME/UFG, Avenida Esperan\c{c}a, s/n, Campus Samambaia, Goi\^ania, GO, 74690-900, Brazil (\email{orizon@ufg.br}).
    }
    \and
    Elianderson M. Santos\thanks{%
        Instituto Federal de Educa\c{c}\~{a}o, Ci\^{e}ncia e Tecnologia do Maranh\~{a}o, Barra do Corda, MA, 65950-000, Brazil (\email{elianderson.santos@ifma.edu.br},\orcid{0000-0003-0936-9640}).
    }
    \and
    João Carlos de O. Souza%
    \thanks{%
        Department of Mathematics, Federal University of Piau\'{i}, Teresina, PI, Brazil (\email{joaocos.mat@ufpi.edu.br},\orcid{0000-0003-4053-8211})
    }
}
\title{The difference of convex algorithm\\ on Hadamard manifolds}

\begin{document}
\maketitle
\begin{abstract}
In this paper, we propose a Riemannian version of the difference of convex algorithm (DCA) to solve a minimization problem involving the difference of convex (DC) function. We establish the equivalence between the classical and simplified Riemannian versions of the DCA. We also prove that, under mild assumptions, the Riemannian version of the DCA is well-defined, and every cluster point of the sequence generated by the proposed method, if any, is a critical point of the objective DC function. Additionally, we establish some duality relations between the DC problem and its dual. To illustrate the effectiveness of the algorithm, we present some numerical experiments.
\end{abstract}
\begin{keywords}
    DC programming,  DCA,   Fenchel conjugate function,  Riemannian manifolds
\end{keywords}
\begin{AMS}
    \href{https://mathscinet.ams.org/msc/msc2010.html?t=90C30}{90C30}
    \href{https://mathscinet.ams.org/msc/msc2010.html?t=90C26}{90C26}
    \href{https://mathscinet.ams.org/msc/msc2010.html?t=49N14}{49N14}
    \href{https://mathscinet.ams.org/msc/msc2010.html?t=49Q99}{49Q99}
\end{AMS}

\section{Introduction}
In this paper, we consider a general non-convex and non-smooth constrained optimization problem
involving a difference of convex functions (shortly, \emph{DC problem}) as follows
\begin{equation}\label{eq:DCOptP}
    \argmin_{p\in \mathcal M}  f(p), \qquad \text{where } f(p)\coloneqq g(p)-h(p),
\end{equation}
where the constrained set $\mathcal M $ is endowed with a structure of a
\emph{complete, simply connected Riemannian manifold of non-positive sectional curvature,
i.e., a Hadamard manifold}, the functions $g,h\colon\mathcal M \to \overline{\mathbb{R}}$,
are convex, lower semi-continuous and proper functions (called DC components), and
$\overline{\mathbb{R}}\coloneqq {\mathbb{R}}\cup\{+\infty\}$ is the extended real line.

Due to the increasing number of optimization problems arising from practical applications
posed in a Riemannian setting, the interest in this topic has increased significantly over
the years. Even though we are not currently concerned with practical issues at this point, we emphasize
that practical applications arise whenever the natural structure of the data is modeled as
an optimization problem on a Riemannian manifold. For example, several problems in image
processing, computational vision and signal processing can be modeled as problems in this
setting. Papers dealing with this subject include~\cite{BacakBergmannSteidlWeinmann2016,%
BergmannPerschSteidl2016,BergmannWeinmann2016, BrediesHollerStorathWeinmann2018,%
WeinmannDemaretStorath2014, WeinmannDemaretStorath2016}, and problems in medical imaging
modeled in this context are addressed in \cite{EspositoWeinmann2019}.
Problems of tracking, robotics and scene motion analysis are also posed on
Riemannian manifolds, as seen in~\cite{FreifeldBlack2012, ParkBobrowPloen1995}. Machine learning \cite{NickelKiela2018},
artificial intelligence~\cite{MuscoloniThomasCiucciBianconiVitt2017},
neural circuits \cite{Sharpee2019}, low-rank approximations of hyperbolic
embeddings~\cite{JawanpuriaMeghwanshiMishra2019,TabaghiDokmanic2020},
Procrustes problems~\cite{TabaghiDokmanic2021},
financial networks~\cite{Keller-ResselNargang2021},
complex networks~\cite{KrioukovPapadopoulosKitsakVahdatBoguna2010, MoshiriSafaeiSamei2021},
embeddings of data~\cite{WilsonHancockPekalskaDuin2014}
and strain analysis \cite{Vollmer2018,Yamaji2008} are some of the other areas where optimization problems on Riemannian manifolds are prevalent.
Additionally, we mention that there are many papers on statistics in the Riemannian context, as seen in \cite{BhattacharyaBhattacharya2008, Fletcher2013}.

As previously mentioned, there has been a significant increase in the number of works focusing on concepts and techniques of
nonlinear programming and convex analysis in the Riemannian setting, see~\cite{AbsilMahonySepulchre2008,Udriste1994}.
In addition to the theoretical issues addressed, which have an interest of their own,
the Riemannian machinery provides support to design efficient algorithms to solve
optimization problem in this setting; papers on this subject include~\cite{AbsiBakerGallivan2007,%
EdelmanAriasSmith1999, HuangGallivanAbsil2015, LiMordukhovichWangYao2011, Manton2015,%
MillerMalick2005, NesterovTodd2002, Smith1994, WenYin2012, WangLiWangYao2015} and references therein.

Recently, the concept of the conjugate of a convex function was introduced in the Riemannian context. This is an important tool in convex analysis and plays a significant role in the theory of duality on Riemannian manifolds,
see~\cite{SilvaLouzeiroBergmannHerzog2022,BergmannHerzogSilvaLouzeiroTenbrinckVidalNunez2021}.
In particular, this definition enables us to propose a Riemannian version of the DCA.

DC problems cover a broad class of non-convex optimization problems and DCA was the first method introduced especially for the standard DC problem \cref{eq:DCOptP}.
It was proposed by~\cite{PhamSouad1986} in the Euclidean setting.
The basic idea behind DCA is to compute a subgradient of each (convex) DC component separately,
i.e., at each iterate $k$, DCA calculates $y^{(k)}\in \partial h(x^{(k)})$ and uses this trial point
to compute $x^{(k+1)} \in \partial g^* (y^{(k)})$, where $\partial g^*$ denotes
the subdifferential of the conjugate function of $g$ in the sense of convex analysis.
Equivalently, DCA approximates the second DC component $h(x)$ by its affine minorization
$h_k(x)=h(x^{(k)}) + \langle x-x^{(k)}, y^{(k)} \rangle$, with $y^{(k)}\in \partial h(x^{(k)})$,
and minimizes the resulting convex function
\begin{equation*}
    x^{(k+1)}\in \argmin_{x\in\mathbb R^n} g(x) - h_ k(x),
\end{equation*}
which is called the alternative version of DCA therein.
On the other hand, computing $y^{(k)}\in \partial h(x^{(k)})$ is equivalent to find a solution of
the dual problem
\begin{equation*}
    \argmin_{y\in\mathbb R^n} h^*(y) - g^*(y^{k-1}) - \langle y-y^{k-1}, x^{(k)} \rangle.
\end{equation*}
Therefore, DCA can also be viewed as an iterative primal-dual subgradient method.

DC optimization algorithms have been proved to be particularly successful for analyzing
and solving a variety of highly structured and practical problems;
see for instance~\cite{Oliveira2020,LeThiPhamDinh2018,AnTao2005}.
To the best of our knowledge, the work in~\cite{SouzaOliveira2015} was the first to deal with DC functions in Riemannian manifolds. More precisely, the authors proposed the proximal point algorithm for DC functions (DCPPA) and studied the convergence of the method in Hadamard manifolds.
Recently, \cite{AlmeidaNetoOliveiraSouza2020} proposed a modified version
of the DCPPA in the same Riemannian setting in order to accelerate
the convergence of the method considered in~\cite{SouzaOliveira2015}.

The aim of this paper is to propose, for the first time, a Riemannian version of the DCA.
We obtain an equivalence between the classical and a simplified version of the Riemannian DCA.
Therefore, under mild assumptions, we prove that the Riemannian DCA is well-defined, and
every cluster point of the sequence generated by the proposed method, if any,
is a critical point of the objective DC function in \cref{eq:DCOptP}.
We also extend some relations between the DC problem \cref{eq:DCOptP} and
its dual to the Riemannian setting. To illustrate the effectiveness of DCA, we present some numerical experiments.

This paper is organized as follows. In \cref{sec:Preliminaries} we present some notations and preliminary results that will
be used throughout the paper. In \cref{sec:Duality} some relations between the DC problem and its dual
are established on Hadamard manifolds. In \cref{sec:DCA_manifolds} we present a formulation of the Riemannian DCA.
In \cref{sec:Convergence} we study the convergence properties of the proposed method.
In \cref{Sec:Numerics} we provide some applications to the problem of maximizing a convex function in a compact set and manifold-valued image denoising. Finally, \cref{sec:Conclusion} presents some conclusions.

\section{Preliminaries}
\label{sec:Preliminaries}
In this section, we recall some concepts, notations, and basics results about Riemannian manifolds and optimization.
For more details see, for example, \cite{doCarmo1992, Rapcsak1997, Sakai1996, Udriste1994}.
Let us begin with concepts about Riemannian manifolds.
We denote by $\mathcal M$ a finite dimensional Riemannian manifold and
by $T_p\mathcal M$ the \emph{tangent space} of $\mathcal M$ at $p$.
The corresponding norm associated to the Riemannian metric $\langle \cdot , \cdot \rangle$
is denoted by $\lVert\cdot\rVert$.
Moreover, the \emph{tangent bundle} of $\mathcal M$,
will be denoted  by $T\mathcal M$. We use $\ell(\gamma)$ to denote the length of a piecewise smooth curve
$\gamma\colon [a,b] \rightarrow \mathcal M$.
The Riemannian distance between $p$ and $q$ in $\mathcal M$ is denoted by $d(p,q)$,
which induces the original topology on $\mathcal M$, namely,
$(\mathcal M, d)$, which is a complete metric space.
A complete, simply connected Riemannian manifold of non-positive sectional curvature
is called a Hadamard manifold. \emph{All Riemannian manifold considered in this paper will be Hadamard manifolds and will be denoted ${\mathcal M}$.} For a $p\in\mathcal M$, the exponential map $\exp_p\colon T_p \mathcal M \to \mathcal M$ is a diffeomorphism
and $\exp^{-1}_p\colon\mathcal M\to T_p\mathcal M$ denotes its inverse.
In this case, $d(q, p) = \lVert \exp^{-1}_pq\rVert$ holds,
the function $d_q^2\colon\mathcal M\to\mathbb{R}$ defined by $d_q^2(p)\coloneqq d^2(q,p)$
is $C^{\infty}$ and its gradient is given by $\grad d_q^2(p) = -2\exp^{-1}_pq$.
Now, we recall some concepts and basic properties about optimization in the Riemannian context.
For that, given two points $p,q\in\mathcal M$,  denotes by $\gamma_{pq}$ the geodesic segment  joining  $p$ to $q$, i.e.,
$\gamma_{pq}\colon[0,1]\rightarrow\mathcal M$ with $\gamma_{pq}(0)=p$ and $\gamma_{pq}(1)=q$. We denote by $\overline{\mathbb{R}}\coloneqq \mathbb{R}\cup\{+\infty\}$ the extended real line.
The \emph{domain} of a function $f\colon\mathcal M \to \overline{\mathbb{R}}$ is denoted by $\dom (f) \coloneqq \{ p\in \mathcal M\ : \ f(p) < +\infty\}$.
The function $f$ is said to be \emph{convex (resp. strictly convex)} if, for any $p,q\in \mathcal M$,
the composition $f\circ{\gamma_{pq}}\colon[0, 1]\to\mathbb{R}$ is convex (resp. strictly convex), i.e.,
$(f\circ{\gamma_{pq}})(t)\leq(1-t)f(p)+tf(q)$ (resp. ($f\circ{\gamma_{pq}})(t)<(1-t)f(p)+tf(q)$),
for all $t\in[0,1]$.
A function $f\colon\mathcal M \to \overline{\mathbb{R}}$ is said to be
\emph{$\sigma$-strongly convex} for $\sigma > 0$ if, for any
$p,q\in \mathcal M$, the composition $f\circ{\gamma_{pq}}\colon[0, 1]\to \overline{\mathbb{R}}$ is
$\sigma$-strongly convex, i.e.,
$(f\circ{\gamma_{pq}})(t)\leq(1-t)f(p)+tf(q)-\frac{\sigma}{2}t(1-t)d^2(q,p)$, for all $t\in[0,1]$.

\begin{definition}
    The \emph{subdifferential} of a proper, convex function
    $f\colon\mathcal M \to \overline{\mathbb{R}}$ at $p\in \mathcal \dom (f) $ is the set
    \begin{equation*}
        \partial f(p)
        \coloneqq
        \bigl\{
            X \in T_p\mathcal M\ : \ %
            f(q) \geq f(p)+\langle X,\exp^{-1}_pq \rangle,
            \quad\text{ for all } q\in \mathcal M
        \bigr\}.
    \end{equation*}
\end{definition}

The proof of the first item of the following theorem can be found in \cite[Theorem 4.10, p. 76]{Udriste1994}, while the proof of the second one follows the same idea as the first one.
\begin{theorem}
    \label{thm:f-convex-subdiff}
    Let $f\colon\mathcal M\to \mathbb{R}$ be a function. Then,
    \begin{enumerate}
        \item
        \label{thm:f-convex-subdiffi}
        The function $f$ is convex if and only if
        $f(p)\geq f(q) + \langle X, \exp^{-1}_qp\rangle$,
        for all $p, q\in \mathcal M$ and all $X\in \partial f(q)$.
        \item
        \label{thm:f-convex-subdiffii}
        The function $f$ is $\sigma$-strongly convex if and only if
        $f(p)\geq f(q) + \langle X, exp^{-1}_q p \rangle + \frac{\sigma}{2}d^2(p,q)$,
        for all $p, q\in \mathcal M$ and all $X\in \partial f(q)$.
    \end{enumerate}
\end{theorem}
The following definition play an import role in the paper, see \cite[p. 363]{Bourbaki1995}.
\begin{definition} \label{def:linf}
    A function $f\colon\mathcal M\to \overline{\mathbb{R}}$ is said to be
    \emph{lower semi-continuous} (\emph{lsc}), at $p\in \mathcal M$ if  $ \liminf_{q\to p} f(q)= f(p)$.  If $f$ is lower semi-continuous at all points along $\mathcal M$, we simply state that $f$ is \emph{lower semi-continuous}.
\end{definition}
The proof of the following result is an immediate consequence of   \cite[Proposition 2.5]{WangLiWangYao2015}.
\begin{proposition}
    \label{cont_subdif}
    Let $f\colon \mathcal M \rightarrow \overline{\mathbb{R}}$ be a convex
    and lower semi-continuous function. Consider the sequence
    $(p^{(k)})_{k\in\mathbb{N}}\subset \intdom (f)$ such that $\displaystyle\lim_{k\to\infty}p^{(k)}={\bar p} \in \intdom (f)$.
    If $(X^{(k)})_{k\in\mathbb{N}}$ is a sequence such that $X^{(k)}\in \partial f(p^{(k)})$
    for every $k\in \mathbb{N}$, then $(X^{(k)})_{k\in\mathbb{N}}$ is bounded and
    its cluster points belongs to $\partial f({\bar p}).$
\end{proposition}
\begin{definition}
    A function $f\colon\mathcal M \to \overline{\mathbb{R}}$ is said to be 1-\emph{coercive} if there exists a point   ${\bar p}\in \mathcal M$ such that
    \begin{equation*}
        \lim_{d({\bar p},p)\to+\infty} \frac{{f(p)}}{{d({\bar p},p)}} = +\infty.
    \end{equation*}
\end{definition}
The \emph{global minimizer set} of a function $f\colon\mathcal M\to \overline{\mathbb{R}}$  is defined by
\begin{equation*}
\Omega^*:=\{q\in {\mathcal M}\ :\ f(q)\leq f(p), \text{for all } p\in {\mathcal M}\}.
\end{equation*}

\begin{proposition}\label{prop:coercive}
    Assume that $f\colon\mathcal M \to \overline{\mathbb{R}}$ is lsc and 1-\emph{coercive}.  Then the global minimizer set of $f$  is non-empty.
\end{proposition}
\begin{proof}
    Take  ${\bar p}\in \mathcal M$ such that    $\lim_{d({\bar p},p)\to+\infty} ({f(p)}/{d({\bar p},p)})= +\infty$. In particular, we  conclude that  $\displaystyle\lim_{d({\bar p},p)\to+\infty } {f(p)} = +\infty$.
    Thus, there exists ${\bar r}>0$ such that
    ${\bar r}<d({\bar p},p)$ implies that $f({\bar p})\leq {f(p)}$.
    Consider the set $B[{\bar p}, {\bar r}]\coloneqq  \{ p\in \mathcal M\ :\ d(p, {\bar p})\leq {\bar r}\}$.
    Since $\mathcal M$ is a Hadamard manifold, the Hopf-Rinow theorem ensures that
    $B[{\bar p}, {\bar r}]$ is compact.
    Thus, taking into account that $f$ is lsc., by~\cite[Theorem~3, p.~361]{Bourbaki1995}
    there exists ${\hat p}\in B[{\bar p}, {\hat r}]$ such that $f({\hat p})\leq f(p)$,
    for all $p\in B[{\bar p}, {\hat r}]$.  Therefore,  ${\bar p}$ or ${\hat p}$ is a  global minimizer of $f$.
\end{proof}
\begin{lemma}\label{L:coercrf}
    Let $g\colon\mathcal M \to\mathbb{R}$ be a $\sigma$-strongly convex function.
    Take ${\bar p}\in \mathcal M$ and $X\in T_{{\bar p}}\mathcal M$.
    Then, the function $f\colon\mathcal M \to\mathbb{R}$ defined by
    $f(p)=g(p)-\big\langle X, \exp^{-1}_{\bar p}p \big\rangle$
    is 1-coercive. Consequently, the global minimizer set of $f$ is non-empty.
\end{lemma}
\begin{proof}
    Since the function $g\colon \mathcal M \to\mathbb{R}$ is a $\sigma$-strongly convex, \cref{thm:f-convex-subdiff} \cref{thm:f-convex-subdiffi} implies that
    \begin{equation*}
        g(p)\geq g({\bar p})
        + \langle {Y},\exp^{-1}_{\bar p}p\rangle
        + \frac{\sigma}{2}d^2({\bar p},p),
        \qquad \text{for all } p \in \mathcal M \text{ and all } {Y}\in \partial g({\bar p}).
    \end{equation*}
    Thus, considering that $f(p)=g(p)-\big\langle X, \exp^{-1}_{\bar p}p \big\rangle$
    and using the last inequality we conclude
    \begin{align*}
        \frac{f(p)}{d({\bar p},p)}
        \geq
        \frac{g({\bar p})}{d({\bar p},p)}
        + \Big\langle {Y},\frac{\exp^{-1}_{\bar p}p}{d({\bar p},p)}\Big\rangle
        + \frac{\sigma}{2}d({\bar p},p)
        - \Big\langle X,\frac{\exp^{-1}_{\bar p}p}{d({\bar p},p)}\Big\rangle,
        \qquad\text{ for all } {Y}\in \partial g({\bar p}).
    \end{align*}
    Since $d({\bar p},p)=\,\lVert\exp^{-1}_{\bar p}p\rVert$,
    we obtain that the inner products in the last inequality are bounded. Hence, we have
    \[
        \lim_{d({\bar p},p)\to+\infty }\frac {f(p)}{d({\bar p},p)}=+\infty.
    \]
    Therefore, $f$ is 1-coercive. The second part of the proposition is an immediate consequence of the first one combined with \cref{prop:coercive}.
\end{proof}
The statement and proof of the next proposition can be found
in~\cite[Lemma 2.4, p. 666]{LiLopezMartinMarquez2009}.
\begin{proposition}\label{pr:cont_pi}
    Let ${\bar p}\in \mathcal M$ and $(p^{(k)})_{k\in \mathbb{N}}\subset \mathcal M$
    be such that $\displaystyle\lim_{k\to +\infty}p^{(k)}= {\bar p}$.
    Then the following assertions hold:
\begin{enumerate}
    \item
    \label{pr:cont_pi_i}
    For any $p\in \mathcal M$, we have
    $\displaystyle\lim_{k\to +\infty}\exp_{p^{(k)}}^{-1}p = \exp_{{\bar p}}^{-1}p$
    and $\displaystyle\lim_{k\to +\infty} \exp_p^{-1}p^{(k)} = \exp_p^{-1}{\bar p}$.
    \item
    \label{pr:cont_pi_ii}
    If $X^{(k)}\in T_{p^{(k)}}\mathcal M$ and $\displaystyle\lim_{k\to +\infty}X^{(k)}= {\bar X}$,
    then ${\bar X}\in T_{{\bar p}}\mathcal M$.
    \item
    \label{pr:cont_pi_iii}
    Given $X^{(k)}\in T_{p^{(k)}}\mathcal M$, $Y^{(k)}\in T_{p^{(k)}}\mathcal M$,
    ${\bar X}\in T_{{\bar p}}\mathcal M$, and ${\bar Y}\in T_{{\bar p}}\mathcal M$.
    If $\displaystyle\lim_{k\to +\infty}X^{(k)}= {\bar X}$
    and $\displaystyle\lim_{k\to +\infty}Y^{(k)}= {\bar Y},$
    then $\displaystyle\lim_{k\to +\infty}\langle X^{(k)}, Y^{(k)} \rangle
     = \langle {\bar X},{\bar Y} \rangle$.
\end{enumerate}
\end{proposition}
We end this section recalling   several results from Fenchel duality on Hadamard manifolds, which play an important role in the following sections. It is worth emphasizing that we are limiting our study to finite-dimensional manifolds and that our emphasis is algorithmic.   Consequently, we do not need to use cotangent space for our purposes.  Due to this, we decided to exclusively employ tangent spaces in the following  results of the paper~\cite{SilvaLouzeiroBergmannHerzog2022}.  We begin by  recalling  the defining the conjugate of a proper function.
\begin{definition}\label{def:conj_nova}
Let $f\colon\mathcal M \to \overline{\mathbb{R}}$ be a proper function.  The Fenchel conjugate of  $f$ is
    the function $f^{*} \colon T\mathcal M\to \overline{\mathbb{R}}$
    defined by
    \begin{equation*}
    f^{*} (p,X)\coloneqq \sup_{q\in \mathcal M}
        \bigl\{ \langle X, \exp^{-1}_pq \rangle - f(q) \bigr\},
        \qquad (p,X) \in T\mathcal M.
    \end{equation*}
\end{definition}
\begin{theorem}\label{th:FYIN}
    Let $f\colon\mathcal M \to \overline{\mathbb{R}}$ be a proper function.
    Then, the \emph{Fenchel-Young inequality} holds, i. e.,
    for all $(p,X)\in T\mathcal M$ we have
    \begin{equation*}
    f(q)+f^{*}(p,X)\geq \langle X, \exp^{-1}_pq \rangle,
        \qquad\text{for all } q \in \mathcal M.
    \end{equation*}
\end{theorem}
\begin{theorem}
    Let $f\colon \mathcal M\to \overline{\mathbb{R}}$ be a proper lsc  convex function
    and $p\in \mathcal M$.
    Then the function
    $f^{*}(p, \cdot)\colon T_p\mathcal M\to \overline{\mathbb{R}}$
    is convex and proper.
\end{theorem}
\begin{definition}\label{def:subdif_novo}
    Let $p \in \mathcal M$ and suppose that $f^{*}(p, \cdot)\colon T_p\mathcal M\to \overline{\mathbb{R}}$ is proper. The subdifferential of $f^{*}(p,\cdot )$ at $X \in T_p\mathcal M$,
    denoted by $\partial_2f^{*}(p, X)$, is the set
    \begin{equation*}
        \partial_2f^{*}(p, X)
        = \bigl\{
            Y\in T_p\mathcal M ~:~ f^{*}(p,Z)\geq f^{*}(p,X)
            + \langle Z-X, Y \rangle, \quad \text{ for all }\quad
            Z\in T_p\mathcal M
        \bigr\}.
    \end{equation*}
\end{definition}
Combining~\cite[Remark 3.3]{SilvaLouzeiroBergmannHerzog2022}
with~\cite[Corollary 3.16]{BergmannHerzogSilvaLouzeiroTenbrinckVidalNunez2021},
we obtain the following result.
\begin{theorem}\label{th:FYINPC}
    Let $f\colon\mathcal M\to\overline{\mathbb{R}}$ be a proper function and $p\in \mathcal M$.
    Then, $Y\in \partial_2f^{*}(p, X)$ if and only if
    $f(\exp_p Y)+f^{*}(p,X)= \langle X, Y \rangle$.
\end{theorem}
\begin{remark}[{\cite[Remark 3.4]{SilvaLouzeiroBergmannHerzog2022}}]
    \label{rmk:subdif}
    If
    $\mathcal M=\mathbb{R}^{n}$, then $f^{*}(p,X)=f^{*}(X)-\langle X,p \rangle$.
    Moreover,
    $\partial_2f^{*}(p,X) \coloneqq \partial f^{*}(X)+\{-p\}$.
\end{remark}
\begin{definition}\label{def:bcconj_nova}
    The Fenchel biconjugate of a function $f\colon \mathcal M\to \overline{\mathbb{R}}$ is the function
    $f^{**}\colon\mathcal M\to \overline{\mathbb{R}}$ defined by
    \begin{equation*}
    f^{**} (p)\coloneqq \sup_{(q,X)\in T\mathcal M}
        \Bigl\{
            \langle X, \exp^{-1}_pq \rangle - f^*(q, X)
        \Bigr\},
        \qquad \text{ for all } p \in \mathcal M.
    \end{equation*}
\end{definition}
\begin{theorem}
    \label{th:bcec}
    Let $f\colon \mathcal M\to \overline{\mathbb{R}}$ be a proper lsc convex function.
    Then, $f^{**}=f$ holds.
\end{theorem}
\begin{theorem}\label{th:eqsubdifcf}
    Let $f\colon \mathcal M\to \overline{\mathbb{R}}$ be a proper convex function.
    Then,
    \begin{equation*}
       X \in \partial f(p)\text{ if and only if }f^{*}(p,{X})=-f(p).
    \end{equation*}
\end{theorem}
\section{Duality in DC optimization in Hadamard manifolds}
\label{sec:Duality}
In this section our aim is to state and study difference of convex optimization problem or DC problem, and its dual problem, called dual DC problem, in the Hadamard setting. The DC problem is defined as follows
\begin{equation}\label{Pr:DCproblem}
\argmin_{p\in \mathcal M} f(p), \qquad \text{where } f(p)\coloneqq g(p)-h(p),
\end{equation}
and $g\colon\mathcal M \to \overline{\mathbb{R}}$ and $h\colon\mathcal M \to \overline{\mathbb{R}}$ are proper, lsc and convex functions.
The DC problem is a non-convex and, in general, a non-smooth problem. In the following we further use the conventions
\begin{align}
    \label{eq:conventions}
    (+\infty)-(+\infty) = +\infty,\quad
    (+\infty)-\lambda = +\infty, \text{ and }
    \lambda- (+\infty) = -\infty,
    \qquad \text{for all } \lambda \in \mathbb{R}.
\end{align}
Similarly to Euclidean context, see \cite{TaoSouad1988}, the dual DC problem of the problem~\eqref{Pr:DCproblem}, is stated as follows
\begin{equation}\label{Pr:Dual}
\argmin_{(p,X) \in T\mathcal M}
    \varphi (p, X),
    \qquad\text{where } \varphi (p, X)
        \coloneqq h^{*}(p,X)-g^{*}(p,X).
\end{equation}
Additional detail concerning the appropriateness of the previous definition will be provided later in Theorem~\ref{equiv:dualpr}. In the following remark, we will look at the details of the relationship between \eqref{Pr:Dual} and its Euclidean counterpart.
\begin{remark}
    \label{rem:DCPonRn}
    If $\mathcal M=\mathbb{R}^{n}$, then $T_p\mathcal M \simeq \mathbb{R}^{n}$
    for all $p\in \mathcal M$.
    Consequently, $T\mathcal M \simeq \mathbb{R}^{n}$.
    Moreover, by using \cref{rmk:subdif}, we obtain
    \begin{align*}
        h^{*}(p,X)-g^{*} (p,X)
        = h^{*}(X)-\langle X,p \rangle - \left( g^{*}(X)-\langle X,p \rangle \right)
        = h^{*}(X)-g^{*}(X), \quad \text{for all }  X \in \mathbb{R}^n.
    \end{align*}
    Therefore, if $\mathcal M=\mathbb{R}^{n}$ then problem~\eqref{Pr:Dual} simplifies to
    \begin{equation}\label{Pr:Dual_euclidean}
        \argmin_{X \in \mathbb{R}^{n} }  h^{*}(X)-g^{*} (X).
    \end{equation}
    In conclusion, for $\mathcal M=\mathbb{R}^{n}$, the dual \eqref{Pr:Dual} of the problem~\eqref{Pr:DCproblem} merges into the dual stated in \cite{TaoSouad1988}.
\end{remark}
To proceed with the study of problems~\eqref{Pr:DCproblem} and~\eqref{Pr:Dual},
for now on we will assume that:
\begin{enumerate}[label={A\arabic*)}, ref={(A\arabic*)}]
    \item%
    \label{it:A1} $g\colon\mathcal M \to \overline{\mathbb{R}}$
        and $h\colon\mathcal M \to \overline{\mathbb{R}}$ are $\sigma$-strongly convex
        and lsc functions, where $\sigma>0$;
    \item%
    \label{it:A2} $f_{\inf}\coloneqq \displaystyle\inf_{x\in\mathcal M}
         f(x) > -\infty$;
    \item%
    \label{it:A3} $\dom (g) \subseteq \intdom(h);$
    \item%
    \label{it:A4} $\partial_2g^{*}(p,X)\neq \varnothing$,
        for every $X \in \dom (g^{*}(p,\cdot))%
            \coloneqq \{ X \in T_p\mathcal M\ : \ g^{*}(p,X)<+\infty \}$.
    \end{enumerate}
Next, we discuss the above assumptions. First, we show that \ref{it:A1} is not restrictive.
\begin{remark}
    Let $q \in \mathcal M$ and $\sigma >0$.
    Consider the function $\mathcal M\ni p \mapsto \frac{\sigma}{2}d^{2}(q, p)$,
    which is $\sigma$-strongly convex, see~\cite[Corollary 3.1]{NetoFerreiraPerez2002}.
    If $\tilde g \colon\mathcal M\to \mathbb{R}$
    and $\tilde h \colon\mathcal M\to \mathbb{R}$ are convex,
    then taking $q\in \mathcal M$ and
    setting $g(p)={\tilde g}(p) + \frac{\sigma}{2}d^{2}(q, p)$
    and $h(p)={\tilde h}(p)+ \frac{\sigma}{2}d^{2}(q, p)$ we obtain two $\sigma$-strongly
    convex functions $g$ and $h$ in $\mathcal M$.
    In addition, $f (p)={\tilde g}(p)-{\tilde h}(p)=g(p)-h(p)$, for all $p\in \mathcal M$.
\end{remark}
\begin{remark}\label{rmk:1}
    If assumption \ref{it:A2} holds, then $\dom (f)=\dom (g)\subseteq \dom (h)$.
    Indeed, if $\dom (g)\nsubseteq \dom (h),$ then there exists $p\in \dom (g)$
    such that $p\notin \dom (h)$, and hence by \eqref{eq:conventions}, we have that
    $f(p)=g(p)-h(p)=g(p)-(+\infty)=-\infty$, which contradicts assumption \ref{it:A2}. Thus, $\dom (g)\subseteq \dom (h)$, which implies that $\dom (g)\subseteq \dom (f)$.
    On the other hand, assume by contradiction that $\dom (f)\nsubseteq \dom (g)$.
    Then, there exists $p\in \dom (f)$ such that $g(p)=+\infty$.
    From \eqref{eq:conventions} we obtain that $f(p)=g(p)-h(p)=(+\infty)-h(p)=+\infty$,
    which contradicts the fact that $p\in \dom (f)$.
    Therefore,  we conclude that $\dom (f)=\dom (g)$. Since under assumption \ref{it:A2},  we have $\dom (f)=\dom (g)$, which implies that  $\dom (g)\subseteq \dom (h)$.  Hence, assumption \ref{it:A3} is only slightly more restrictive than assumption \ref{it:A2}.
    We also note that if $\dom (h)=\mathcal M$, then assumption~\ref{it:A3} holds,
    and if $\dom (g^{*}(p,\cdot)) = T_p\mathcal M$,
    then assumption~\ref{it:A4} holds.
    It is worth to note that assumption~\ref{it:A4} is used here to establish
    the relationship between problems~\eqref{Pr:DCproblem} and~\eqref{Pr:Dual}.
\end{remark}

A necessary condition for the point $p^{*}\in \mathcal M$ to be a local minimum
of $f = g-h$ is that
$0\in \partial f (p^{*})\subset \partial g(p^{*})-\partial h(p^{*})$.
Hence, if $p^{*}\in \mathcal M$ is the solution of problem~\eqref{Pr:DCproblem},
then $\partial h(p^{*}) \subset\partial g(p^{*})$.
Consequently, $\partial g(p^{*}) \cap \partial h(p^{*}) \neq\varnothing$.
In this sense, we define a~\emph{critical point} of problem~\eqref{Pr:DCproblem}.
\begin{definition}%
    \label{def:critpoint}
    A point $p^{*}\in \mathcal M$ is a critical point of $f$ in~\eqref{Pr:DCproblem}
    if $\partial g(p^{*}) \cap \partial h(p^{*}) \neq\varnothing$.
\end{definition}
The next lemma establishes a necessary condition for a point
$({\bar p},{\bar X})\in T\mathcal M$
be a solution of problem~\eqref{Pr:Dual}.
\begin{lemma}
    \label{le:ncsoldual}
    If $({\bar p},{\bar X})$ is a solution of problem~\eqref{Pr:Dual}, then
    $\partial_2g^{*}({\bar p},{\bar X}) \subseteq \partial_2h^{*}({\bar p},{\bar X})$
    holds.
\end{lemma}
\begin{proof}
    Let $({\bar p},{\bar X})$ be a solution of problem~\eqref{Pr:Dual}.
    Then,
    $h^{*}(p, Y)- g^{*} (p,Y)\geq h^{*}({\bar p},{\bar X}) -g^{*} ({\bar p},{\bar X})$,
    for all $(p,Y) \in T\mathcal M$. Thus, we have
    \begin{align*}
        h^{*}({\bar p},Y)-h^{*}({\bar p},{\bar X})
        \geq
        g^{*} ({\bar p},Y) -g^{*} ({\bar p},{\bar X}),
        \qquad \text{for all}\quad Y \in T_{{\bar p}}\mathcal M.
    \end{align*}
    Take $Z\in \partial_2g^{*}({\bar p},{\bar X})$.
    By \cref{def:subdif_novo}, we have
    $g^{*}({\bar p},Y)-g^{*}({\bar p},{\bar X}) \geq \langle Y-{\bar X}, Z \rangle$,
    for all $Y \in T_{{\bar p}}\mathcal M$,
    which combined with the last inequality yields
    \begin{align*}
        h^{*}({\bar p},Y)-h^{*}({\bar p},{\bar X})
        \geq \langle Y-{\bar X}, Z \rangle,
        \qquad \text{for all}\quad Y \in T_{{\bar p}}\mathcal M.
    \end{align*}
    This implies, by \cref{def:subdif_novo},  that $v\in \partial_2h^{*}({\bar p},{\bar X})$,
    and the statement is proved.
\end{proof}
\begin{remark}
    If $\mathcal M=\mathbb{R}^{n}$, then by using \cref{rmk:subdif} we have
    $\partial_2g^{*}({\bar p},{\bar X}) = \partial g^{*}({\bar X})-\{{\bar p} \}$
    and $\partial_2h^{*}({\bar p},{\bar X}) = \partial h^{*}({\bar X})-\{{\bar p} \}$.
    Thus, from \cref{rem:DCPonRn} and \cref{le:ncsoldual} we conclude that
    if ${\bar X} \in \mathbb{R}^{n}$ is a solution of problem~\eqref{Pr:Dual_euclidean},
    then we have $\partial g^{*}({\bar X})\subseteq \partial h^{*}({\bar X})$,
    which  yields  \cite[Theorem~2.1~(2)]{TaoSouad1988}.
\end{remark}

We have already defined the critical point for the primal problem in \cref{def:critpoint}, so let us continue on dual problem. Please keep in mind that, it follows from \cref{le:ncsoldual}, that if $({\bar p},{\bar X})$ is a solution
of the problem~\eqref{Pr:Dual}, then the set
$\partial_2h^{*}({\bar p},{\bar X})\cap \partial_2g^{*}({\bar p},{\bar X})$
is non-empty. Hence, we define the notion of critical point for the problem~\eqref{Pr:Dual} as follows:
\begin{definition}
    A point $({\bar p},{\bar X})$ is \emph{a critical point} for problem~\eqref{Pr:Dual} if $
        \partial_2h^{*}({\bar p},{\bar X})\cap \partial_2g^{*}({\bar p},{\bar X})
        \neq\varnothing.$
\end{definition}
\begin{remark}
    If $\mathcal M=\mathbb{R}^{n}$, then by using \cref{rmk:subdif} we have
    $\partial_2g^{*}({\bar p},{\bar X}) = \partial g^{*}({\bar X})-\{{\bar p} \}$ and
    $\partial_2h^{*}({\bar p},{\bar X}) = \partial h^{*}({\bar X})-\{{\bar p} \}$.
    Thus, if $({\bar p},{\bar X})$ is a critical point of problem~\eqref{Pr:Dual},
    then there exist
    $Z\in \partial_2h^{*}({\bar p},{\bar X})\cap \partial_2g^{*}({\bar p},{\bar X})$.
    Hence,
    $Z+{\bar p} \in \partial g^{*}({\bar X})\cap \partial h^{*}({\bar X})\neq\varnothing$.
    Therefore, ${\bar X}\in \mathbb{R}^{n}$ is
    a critical point of problem~\eqref{Pr:Dual_euclidean}.
\end{remark}
To proceed with our analysis we need the next lemma. For a proof of it see \cite[p. 46]{BartleSherbert2000}.
\begin{lemma}\label{le:iterinf}
    Let $X $ and $Y$ be non-empty sets and $f\colon X\times Y \to \mathbb{R}$ a function.
    Then,  it holds
    \begin{align*}
        \inf_{(x,y)\in X\times Y} f(x,y)
        = \inf_{x\in X} \inf_{y \in Y} f(x,y)
        = \inf_{y \in Y} \inf_{x \in X} f(x,y).
    \end{align*}
\end{lemma}
The next theorem presents the relation between the optimum values of
problems~\eqref{Pr:DCproblem} and~\eqref{Pr:Dual}.
\begin{theorem}
    \label{equiv:dualpr}
  Let   $g\colon\mathcal M \to \overline{\mathbb{R}}$ and $h\colon\mathcal M \to \overline{\mathbb{R}}$  be  proper, lsc and convex functions. Then,  there holds
    \begin{equation*}
        \inf_{(q,X) \in T\mathcal M}
                \Bigl\{ h^{*}(q,X)-g^{*} (q,X) \Bigr\}
        = \inf_{p\in \mathcal M}\left\{ g(p) - h(p) \right\}.
    \end{equation*}
\end{theorem}
\begin{proof}
    Since $h$ is convex, \cref{th:bcec} implies that $h^{**}=h$.
    Thus, using \cref{def:bcconj_nova} we have
    \begin{align*}
     \inf_{p\in \mathcal M} \{g(p)-h(p)\}
     & = \inf \{g(p)-h^{**}(p)\ : \ \in \mathcal M\}
     \\
     & = \inf \Biggl\{
            g(p) - \sup_{(q,X)\in T\mathcal M}
                \Bigl\{
                    \langle X,\exp^{-1}_qp \rangle - h^{*}(q,X)
                \Bigr\}
            \ : \ p\in \mathcal M
        \Biggr\}.
    \end{align*}
  Since $\displaystyle\sup_{(q,X)\in T\mathcal M}\{\langle X,\exp^{-1}_qp \rangle - h^{*}(q,X) \}=-\displaystyle\inf_{(q,X)\in T\mathcal M}\{h^{*}(q,X) -\langle X,\exp^{-1}_qp \rangle \}$, the last equality is equivalent to
    \begin{equation*}
        \inf_{p\in \mathcal M} \{g(p)-h(p)\}
        = \inf_{p\in \mathcal M} \inf_{(q,X)\in T\mathcal M}
        \Bigl\{ g(p) + h^{*}(q,X) -\langle X,\exp^{-1}_qp \rangle \Bigr\},
    \end{equation*}
    which, using \cref{le:iterinf}, can still be expressed equivalently as
    \begin{equation*}
        \inf_{p\in \mathcal M} \{g(p)-h(p)\}
        = \inf_{(q,X)\in T\mathcal M} \inf_{p\in \mathcal M}
            \Bigl\{ g(p) + h^{*}(q,X) -\langle X,\exp^{-1}_qp \rangle \Bigr\}.
    \end{equation*}
    Due to  $\inf_{p\in \mathcal M} \{ g(p) + h^{*}(q,X) -\langle X,\exp^{-1}_qp \rangle \}= h^{*}(q,X) - \sup_{p\in \mathcal M} \langle X,\exp^{-1}_qp \rangle - \{ g(p) \}$, the final equality is as follows
    \begin{equation*}
    \inf_{p\in \mathcal M} \{g(p)-h(p)\}
    = \inf_{(q,X)\in T\mathcal M}
        \Biggl\{
            h^{*}(q,X) - \sup_{p\in \mathcal M}
                \bigl\{ \langle X,\exp^{-1}_qp \rangle - g(p) \bigr\}
        \Biggr\},
    \end{equation*}
    which, by using \cref{def:conj_nova}, yields the desired equality
    and the proof is concluded.
\end{proof}
\begin{theorem}\label{th:emp}
    The following statements hold:
    \begin{enumerate}
        \item
        \label{th:empi}
        If ${\bar p}\in \mathcal M$ is a solution of problem~\eqref{Pr:DCproblem},
        then $({\bar p},{\bar Y})\in T\mathcal M$ is a solution of
        the problem~\eqref{Pr:Dual},
        for all ${\bar Y} \in \partial h({\bar p}) \cap \partial g({\bar p})$.
        \item
        \label{th:empii}
        If $({\bar p},{\bar Y})\in T\mathcal M$ is a solution of
        problem~\eqref{Pr:Dual}, for some
        ${\bar Y} \in \partial h({\bar p}) \cap \partial g({\bar p})$,
        then ${\bar p}\in \mathcal M$ is a solution of problem~\eqref{Pr:DCproblem}.
    \end{enumerate}
\end{theorem}
\begin{proof}
    To prove \cref{th:empi}, assume that ${\bar p}\in \mathcal M$ is a solution of
    problem~\eqref{Pr:DCproblem}.
    Thus, we have $\partial h({\bar p}) \cap \partial g({\bar p}) \neq \varnothing$.
    Let ${\bar Y} \in \partial h({\bar p})\cap \partial g({\bar p})$.
    Since $g$ and $h$ are convex, by \cref{th:eqsubdifcf} we have
    $-g^{*}({\bar p},{\bar Y}) = g({\bar p})$ and
    $ h^{*}({\bar p},{\bar Y}) = -h({\bar p})$, which implies that
    $h^{*}({\bar p},{\bar Y})-g^{*}({\bar p},{\bar Y}) = g({\bar p})-h({\bar p})$.
    Using again that ${\bar p}\in \mathcal M$ is a solution of problem~\eqref{Pr:DCproblem},
    the last equality together with \cref{equiv:dualpr} ensure that $({\bar p},{\bar Y})$
    is a solution of problem~\eqref{Pr:Dual}, and hence, the \cref{th:empi} is proved.
    We proceed to prove \cref{th:empii}. To this end, we assume that $({\bar p},{\bar Y})$
    is a solution of problem~\eqref{Pr:Dual} with
    ${\bar Y}\in \partial h({\bar p}) \cap \partial g({\bar p})$.
    Since $g$ and $h$ are convex and
    ${\bar Y} \in \partial h({\bar p}) \cap \partial g({\bar p})$,
    it follows from \cref{th:eqsubdifcf} that $-g^{*}({\bar p},{\bar Y}) = g({\bar p})$
    and $ h^{*}({\bar p},{\bar Y}) = -h({\bar p})$, which implies
    \begin{equation} \label{eq:eqif}
        g({\bar p})-h({\bar p})=h^{*}({\bar p},{\bar Y})-g^{*} ({\bar p},{\bar Y})
        =\inf_{(p,X) \in T\mathcal M}
            \bigl\{ h^{*}(p,X)-g^{*} (p,X) \bigr\}.
    \end{equation}
    On the other hand, \cref{equiv:dualpr} implies that
    \begin{equation*}
        \inf_{(p,X) \in T\mathcal M}
            \bigl\{ h^{*}(p,X)-g^{*} (p,X)\bigr\}
        = \inf_{q\in \mathcal M} \bigl\{ g(q) - h(q) \bigr\}
        \leq g({\bar p})-h({\bar p})
    \end{equation*}
    Combining the last inequality with \eqref{eq:eqif} yields
    $g({\bar p})-h({\bar p}) = \inf_{q\in \mathcal M}\bigl\{ g(q) - h(q) \bigr\}$.
    Hence, ${\bar p}\in \mathcal M$ is a solution of problem~\eqref{Pr:DCproblem}.
\end{proof}
\section{DCA on Hadamard manifolds}%
\label{sec:DCA_manifolds}
The aim of this section is present an extension of the DCA to Hadamard manifolds.
To this end, we first propose an extension of the classical DCA, which is based
on the Fenchel conjugate introduced in \cref{def:conj_nova}.
As the DCA is dependent on the Fenchel conjugate of the  first component of the objective function, which is in general difficult to compute, we provide a much simpler version of DCA on Hadamard manifolds based on a first-order approximation of the second component.
We also show the well-definition of these algorithms and
their equivalence in the Riemannian setting, such as in the linear setting.
The DCA based on Fenchel conjugate is stated in \cref{Alg:DCA1}, and
the second version in \cref{Alg:DCA2}.

\begin{algorithm}[hbp]
    \caption{The DC Algorithm on Hadamard Manifolds (DCA1)}\label{Alg:DCA1}
    \begin{algorithmic}[1]
        \STATE {Choose an initial point $p^{(0)}\in \dom (g) $. Set $k=0$.}
        \STATE{Take $X^{(k)}\in\partial h(p^{(k)})$, and compute
        \begin{equation}
            \label{eq:DCAS_natural}
            \begin{split}
                Y^{(k)}&\in \partial_2 g^{*}(p^{(k)},X^{(k)}),
                \\
                p^{(k+1)} &\coloneqq \exp_{p^{(k)}}Y^{(k)}.
            \end{split}
        \end{equation}
        }
        \STATE{If $p^{(k+1)} =p^{(k)}$, then STOP and return $p^{(k)}$. Otherwise, go to Step~4.}
        \STATE{Set $k \leftarrow k+1$ and go to Step~2.}
    \end{algorithmic}
\end{algorithm}

As mentioned before, \cref{Alg:DCA1} relies on the computation of the Fenchel conjugate, which can be difficult to compute in practice. However, this algorithm is conceptually useful and can be shown to be is equivalent to more practical and computable algorithm that does not rely on the Fenchel conjugate. The following two results will be used to demonstrate the well-definedness of \cref{Alg:DCA1}.
\begin{lemma}
    \label{le:wedda}
    If $p\in \dom (h)$ and $Y\in\partial h(p)$,
    then $\dom (h^{*}(p,\cdot))\subseteq \dom (g^{*}(p,\cdot))$ and
    \begin{equation*}
        Y \in dom(g^{*}(p,\cdot))
        =\{ X \in T_p\mathcal M\ : \ g^{*}(p,X)<+\infty \}.
    \end{equation*}
    In particular, $\partial_2 g^{*}(p,Y)\neq \varnothing$.
\end{lemma}
\begin{proof}
    Assume that $p\in \dom (h)$ and take $Y\in\partial h(p)$.
    Thus, by using \cref{th:eqsubdifcf} we obtain
    \begin{equation}\label{eq:aaa1}
    h^{*}(p,Y)=-h(p)<+\infty .
    \end{equation}
    From \cref{equiv:dualpr} and assumption~\ref{it:A2} we have that
    \begin{equation}\label{eq:aaa2}
        h^{*}(p,Y)-g^{*}(p,Y)
        \geq \inf_{(q,X) \in T\mathcal M}
            \bigl\{ h^{*}(q,X) -g^{*} (q,X)\bigr\}
        = \inf_{q\in \mathcal M}\bigl\{ g(q) - h(q) \bigr\}>-\infty.
    \end{equation}
    To prove the first statement, assume by contradiction that
    $\dom (h^{*}(p,\cdot))\nsubseteq \dom (g^{*}(p,\cdot))$.
    Thus, there exists ${\bar Y} \in T_p\mathcal M$ such that
    $h^{*}(p, {\bar Y})<+\infty$ and $g^{*}(p, {\bar Y})=+\infty$.
    By using \eqref{eq:conventions}, we have
    $h^{*}(p,{\bar Y})-g^{*}(p,{\bar Y})=h^{*}(p,{\bar Y})-(+\infty)=-\infty$,
    which contradicts the equality in \eqref{eq:aaa2} and the first statement is proved.
    Since $\dom (h^{*}(p,\cdot))\subseteq \dom (g^{*}(p,\cdot))$,
    it follows from~\eqref{eq:aaa1} that $g^{*}(p,Y)<+\infty$.
    Thus, $Y\in \dom (g^{*}(p,\cdot))$ and by assumption~\ref{it:A4}
    we conclude that $\partial_2 g^{*}(p,Y)\neq \varnothing$.
\end{proof}
\begin{proposition}
\cref{Alg:DCA1} is well defined.
\end{proposition}
\begin{proof}
    Assume $p^{(k)}\in \dom (g)$.
    From \cref{rmk:1}, we have that $\dom (f)=\dom (g)\subseteq \dom (h)$,
    and hence $p^{(k)}\in \dom (h)$.
    By assumption~\ref{it:A3}, we have that $\partial h(p^{(k)})\neq \varnothing$.
    Let $X^{(k)}\in \partial h(p^{(k)})$.
    Since $h$ is convex, \cref{th:eqsubdifcf} implies that
    $h^{*}(p^{(k)},X^{(k)})=-h(p^{(k)})<+\infty$.
    By the first part of \cref{le:wedda}, we have that $g^{*}(p^{(k)},X^{(k)})<+\infty$ and
    $\partial_ 2g^{*}(p^{(k)},X^{(k)})\neq \varnothing$.
    Let $Y^{(k)}\in \partial_2 g^{*}(p^{(k)},X^{(k)})$.
    Since $\mathcal M$ is Hadamard, the point $p^{(k+1)}=\exp_{p^{(k)}}Y^{(k)}$ is well defined
    and belongs to $\mathcal M$.
    Moreover, applying \cref{th:FYINPC} with $f=g$, $p=p^{(k)}$, $X=X^{(k)}$ and $Y=Y^{(k)}$ we have
    $g(p^{(k+1)})+g^{*}(p^{(k)},X^{(k)})=\langle X^{(k)},Y^{(k)} \rangle$ or equivalently
    $g(p^{(k+1)})=\langle X^{(k)},Y^{(k)} \rangle -g^{*}(p^{(k)},X^{(k)})<+\infty$,
    which implies that $p^{(k+1)}\in \dom (g)=\dom (f) \subseteq \dom (h)$.
    Therefore, \cref{Alg:DCA1} is well defined.
\end{proof}

In the following, we present a second version of the DCA that is equivalent to \cref{Alg:DCA1}, which is described in \cref{Alg:DCA2}.

\begin{algorithm}[hbp]
    \caption{The DC Algorithm on Hadamard Manifolds (DCA2)}\label{Alg:DCA2}
    \begin{algorithmic}[1]
        \STATE {Choose an initial point $p^{(0)}\in \dom (g) $. Set $k=0$.}
        \STATE{Take $X^{(k)}\in\partial h(p^{(k)})$,
            and the next iterated $p^{(k+1)}$ is define as following
            \begin{equation} \label{eq:DCAS}
                p^{(k+1)}\in \argmin_{p\in \mathcal M}
                    \biggl(
                        g(p)-\big\langle X^{(k)}, \exp^{-1}_{p^{(k)}}p\big\rangle
                    \biggr).
            \end{equation}
        }
        \STATE{If $p^{(k+1)} =p^{(k)}$, then STOP and return $p^{(k)}$. Otherwise, go to Step~4.}
        \STATE{Set $k \leftarrow k+1$ and go to Step~2.}
    \end{algorithmic}
\end{algorithm}
It should be noted that the stopping criterion in step 3 of \cref{Alg:DCA2} allows it to generate an infinite sequence. Therefore, in practice, to implement \cref{Alg:DCA2}, an appropriate stopping criterion will be required, which will be addressed further in the implementation section.  Let us now analyze \cref{Alg:DCA2}. First of all, note that due to the point $ p^{(k+1)}$ be a solution of \eqref{eq:DCAS}, we have
\begin{equation}\label{eq:sol}
    g(p) - \big\langle X^{(k)}, \exp^{-1}_{p^{(k)}}p \big\rangle
    \geq
    g(p^{(k+1)})-\big\langle X^{(k)}, \exp^{-1}_{p^{(k)}}{p^{(k+1)}} \big\rangle,
    \qquad \text{for all } p\in \mathcal M.
\end{equation}
This inequality will now have an important role in the paper.
\begin{proposition}
    \cref{Alg:DCA2} is well defined.
\end{proposition}
\begin{proof}
    Assume that $p^{(k)}\in \dom (g)$.
    From \cref{rmk:1}, we have that $\dom (g)=\dom (f)\subseteq \dom (h)$,
    which implies that $p^{(k)}\in \dom (h)$.
    Thus, by Assumption~\ref{it:A3}, we have that $\partial h(p^{(k)})\neq \varnothing$.
    Let $X^{(k)}\in \partial h(p^{(k)})$.
    From \cref{L:coercrf}, we have that $g_{k}\colon\mathcal M\to \overline{\mathbb{R}}$
    given by $g_{k}(p)\coloneqq g(p)-\langle X^{(k)},\exp^{-1}_{p^{(k)}}p\rangle$
    is 1-coercive.
    Consequently, its minimizer set is non-empty and is contained in $\dom (g)$.
    Therefore, there exists $p^{(k+1)}\in \dom (g)=\dom (f)$ such that
    $p^{(k+1)}\in \argmin_{p\in \mathcal M} (g(p)
    -\langle X^{(k)}, \exp^{-1}_{p^{(k)}}p \rangle)$,
    which implies that \cref{Alg:DCA2} is well defined.
\end{proof}
\begin{remark}
    If $\mathcal M=\mathbb{R}^{n}$, then by \cref{rmk:subdif},
    we have $\partial_2g^{*}(p^{(k)},X^{(k)} ) = \partial g^{*}(X^{(k)})-\{p^{(k)}\}$
    and consequently
    $Y^{(k)}+p^{(k)} = \exp_{p^{(k)}}Y^{(k)} = p^{(k+1)} \in \partial g^{*}(X^{(k)})
        = \partial_2g^{*}(p^{(k)},X^{(k)} )+\{p^{(k)}\}$, i.e.,
    $p^{(k+1)} \in \partial g^{*}(X^{(k)})$ and $X^{(k)}\in \partial h(p^{(k)})$.
    Therefore, \cref{Alg:DCA1} coincides with the classical formulation of the DCA;
    see~\cite{TaoSouad1988,AnTao2005}.
    Moreover, if $\mathcal M=\mathbb{R}^{n}$, then \eqref{eq:sol} becomes
    \begin{equation*}
        g(p) - \big\langle X^{(k)}, p-p^{(k)} \big\rangle
        \geq g(p^{(k+1)})-\big\langle X^{(k)}, p^{(k+1)}-p^{(k)} \big\rangle,
        \qquad \text{ for all } p\in \mathbb{R}^{n},
    \end{equation*}
which is equivalent to $p^{(k+1)}=\argmin_{p \in \mathbb{R}^{n}} \{ g(p)-\big\langle X^{(k)}, p-p^{(k)} \big\rangle\}$
As a conclusion, \cref{Alg:DCA2} yields an alternative version of the classical DCA
\end{remark}
In the next result, we show that \cref{Alg:DCA1} is equivalent to \cref{Alg:DCA2} in the Riemannian setting, similar to the linear setting.
\begin{proposition}
    \label{equivalenceAlg}
    If $p^{(k)}\in \dom (g)$, $X^{(k)}\in \partial h(p^{(k)})$ and
    $Y^{(k)}\in \partial_2g^{*}(p^{(k)},X^{(k)})$,
    then $p^{(k+1)}=\exp_{p^{(k)}}Y^{(k)}$ if and only if
    $p^{(k+1)}\in \argmin_{p\in \mathcal M}
        (g(p)-\langle X^{(k)}, \exp^{-1}_{p^{(k)}}p\rangle)$.
    Consequently, \cref{Alg:DCA1} is equivalent to \cref{Alg:DCA2}.
\end{proposition}
\begin{proof}
    Let $p^{(k)}\in \dom(g)$, $X^{(k)}\in \partial h(p^{(k)})$,
    $Y^{(k)}\in \partial_2g^{*}(p^{(k)},X^{(k)})$, and $p^{(k+1)}=\exp_{p^{(k)}}Y^{(k)}$ be given
    by \cref{Alg:DCA1}.
    By applying \cref{th:FYINPC} with $f=g$, $p=p^{(k)}$, $Y=\exp^{-1}_{p^{(k)}}p^{(k+1)}$,
    and $X=X^{(k)}$, we have
    $g(p^{(k+1)})+g^{*}(p^{(k)},X^{(k)})=\langle X^{(k)}, \exp^{-1}_{p^{(k)}}p^{(k+1)} \rangle$,
    which by using \cref{def:conj_nova} is equivalent to
    \begin{equation*}
        g(p^{(k+1)})-\langle X^{(k)},\exp^{-1}_{p^{(k)}}p^{(k+1)} \rangle
        = -g^{*}(p^{(k)},X^{(k)}) = -\sup_{q\in \mathcal M} \left( \langle X^{(k)},\exp^{-1}_{p^{(k)}}q \rangle - g(q)\right),
    \end{equation*}
    or equivalently,
    \begin{equation*}
        g(p^{(k+1)})-\langle X^{(k)},\exp^{-1}_{p^{(k)}}p^{(k+1)} \rangle
        = \displaystyle\inf_{q\in\mathcal M}
            (g(q)-\langle X^{(k)}, \exp^{-1}_{p^{(k)}}q \rangle).
    \end{equation*}
    This is also equivalent to
    $p^{(k+1)} \in \displaystyle\argmin_{p\in \mathcal M }
        (g(p)-\langle X^{(k)}, \exp^{-1}_{p^{(k)}}p\rangle)$.
    Therefore, \cref{Alg:DCA1} is equivalent to \cref{Alg:DCA2}.
\end{proof}
\section{Convergence analysis of DCA}\label{sec:Convergence}
The aim of this section is to study the convergence properties of DCA.
It is worth mentioning that the results in this section can be proved using either of
the formulations of DCA in \cref{Alg:DCA1} and~\ref{Alg:DCA2}, as they are equivalent
according to \cref{equivalenceAlg}.
For simplicity, we present the results only using \cref{Alg:DCA2}, but the proofs of
the results for \cref{Alg:DCA1} are quite similar.
We begin by showing a descent property of the algorithm.
\begin{proposition}\label{pr:ffr}
    Let $(p^{(k)})_{k\in \mathbb N}$ be generated by \cref{Alg:DCA2}.
    Then, the following inequality holds
    \begin{equation}\label{eq:dsc}
        f(p^{(k+1)})\leq f(p^{(k)})-\frac{\sigma}{2}d^{2}(p^{(k)},p^{(k+1)}).
    \end{equation}
    Moreover, if $p^{(k+1)} =p^{(k)}$, then $p^{(k)}$ is a critical point of $f$.
\end{proposition}
\begin{proof}
    By using inequality in~\eqref{eq:sol} with $p=p^{(k)}$ we have
    $g(p^{(k)})-g(p^{(k+1)}) \geq \langle -X^{(k)}, \exp^{-1}_{p^{(k)}}{p^{(k+1)}}\rangle$.
    On the other hand, since $h$ is $\sigma$-strongly convex and
    $X^{(k)}\in \partial h(p^{(k)})$, we obtain that
    \begin{equation*}
        h(p^{(k+1)})-h({p^{(k)}})
            \geq \langle X^{(k)},\exp^{-1}_{p^{(k)}}p^{(k+1)}\rangle
                +\frac{\sigma}{2}d^2(p^{(k+1)},p^{(k)}).
    \end{equation*}
    Hence, using that $f=g-h$ together with two previous inequalities we
    obtain~\eqref{eq:dsc}. To prove the last statement, we assume that $p^{(k+1)} =p^{(k)}$.
    Thus, \eqref{eq:sol} implies that
    $g(p)\geq g(p^{(k)})+ \langle X^{(k)}, \exp^{-1}_{p^{(k)}}p \rangle$,
    for all $p\in \mathcal M$, which shows that $X^{(k)}\in \partial g(p^{(k)})$.
    Hence, taking into account that $X^{(k)}\in\partial h(p^{(k)})$, we conclude that
    $X^{(k)}\in \partial g(p^{(k)}) \cap \partial h(p^{(k)}) \neq~\varnothing$.
    Therefore, it follows from \cref{def:critpoint} that $p^{(k)}$ is a critical point
    of $f$ in problem~\eqref{Pr:DCproblem}.
\end{proof}
\begin{proposition}\label{pr:consit}
    Let $(p^{(k)})_{k\in \mathbb N}$ be generated by \cref{Alg:DCA2}.
    Then,
    \begin{equation*}
        \sum_{k=0}^{+\infty}d^{2}(p^{(k)},p^{(k+1)})<+\infty.
    \end{equation*}
    In particular, $\displaystyle\lim_{k\to+\infty} d (p^{(k)},p^{(k+1)})=0$.
\end{proposition}
\begin{proof}
    It follows from~\eqref{eq:dsc} that
    $0\leq (\sigma/2)d^{2}(p^{(k)},p^{(k+1)})\leq f (p^{(k)})-f (p^{(k+1)})$,
    for all $k\in \mathbb{N}$. Thus,
    \begin{align*}
        \sum_{k=0}^Td^{2}(p^{(k)},p^{(k+1)})
        \leq \frac{2}{\sigma}\sum_{k=0}^T\left( f (p^{(k)})-f (p^{(k+1)}) \right) \leq \frac{2}{\sigma}\left( f (p^{(0)})-f _{\inf} \right),\\
    \end{align*}
    for each $T\in \mathbb N$, where $f_{\inf}>-\infty$ is given by assumption~\ref{it:A2}. Taking the limit in the last inequality, as $T$ goes to $+\infty$, we obtain
    the first statement. The second statement is an immediate consequence of the first one.
\end{proof}
\begin{theorem}\label{pr:assympc}
    Let $(p^{(k)})_{k\in \mathbb N}$ and $(X^{(k)})_{k\in \mathbb N}$ be generated
    by \cref{Alg:DCA2}. If ${\bar p}$ is a cluster point of $(p^{(k)})_{k\in \mathbb{N}}$,
    then ${\bar p}\in \dom(g)$ and there exists a cluster point ${\bar X}$ of
    $(X^{(k)})_{k\in \mathbb N}$ such that
    ${\bar X}\in \partial g({\bar p})\cap \partial h({\bar p})$.
    Consequently, every cluster point of $(p^{(k)})_{k\in \mathbb{N}}$,
    if any, is a critical point of $f$.
\end{theorem}
\begin{proof}
    Let ${\bar p}\in \mathcal M$ be a cluster point of $(p^{(k)})_{k\in \mathbb{N}}$.
    Without loss of generality we can assume that
    $\displaystyle\lim_{k \to +\infty}p^{(k)}={\bar p}$.
    It follows from \cref{pr:ffr} together with assumption \ref{it:A2} that
    $(f(p^{(k)}))_{k\in \mathbb N}$ is non-increasing and converges.
    Moreover, due to $f(p^{(0)})\geq f(p^{(k)})=g(p^{(k)})-h(p^{(k)})$ and $g$ be lsc, we have
    \begin{equation*}
        f(p^{(0)})\geq \liminf_{k\to +\infty} g(p^{(k)}) - \limsup _{k\to +\infty}h(p^{(k)})
        \geq g({\bar p})- \limsup _{k\to +\infty}h(p^{(k)}).
    \end{equation*}
    Thus, using the convention~\eqref{eq:conventions} we conclude that ${\bar p}\in \dom (g)$.
    Hence, using assumption \ref{it:A3}, we conclude that ${\bar p}\in \intdom(h)$.
    We know that $X^{(k)}\in\partial h(p^{(k)})$, for all $k\in \mathbb{N}$.
    Thus, by \cref{cont_subdif}, we can also conclude that
    $\displaystyle\lim_{k\to+\infty} X^{(k)}={\bar X}\in \partial h({\bar p})$.
    Due to the point $p^{(k+1)}$ being a solution of~\eqref{eq:DCAS},
    it satisfies \eqref{eq:sol}.
    Thus, taking the inferior limit in \eqref{eq:sol}, as $k$ goes to $+\infty$,
    and using the fact that $\displaystyle\lim_{k\to+\infty}p^{(k)}={\bar p}$,
    $g$ is lsc together with \cref{pr:cont_pi},\cref{pr:cont_pi_iii} and \cref{pr:consit},
    we obtain
    \begin{equation*}
        g(p) \geq \liminf_{k\to+\infty}
            \biggl(
                g(p^{(k+1)}) + \langle X^{(k)}, \exp^{-1}_{p^{(k)}}p \rangle
                           - \langle X^{(k)}, \exp^{-1}_{p^{(k)}}p^{(k+1)} \rangle
            \biggr)
            \geq g({\bar p})+\langle {\bar X}, \exp^{-1}_{{\bar p}}p\rangle,
    \end{equation*}
    for each $p\in \mathcal M$, which implies that
    $g(p) \geq g({\bar p})+\langle {\bar X}, \exp^{-1}_{{\bar p}}p\rangle$,
    for all $p\in \mathcal M$.
    Hence, ${\bar X}\in \partial g({\bar p})$.
    Therefore, ${\bar X}\in \partial g({\bar p})\cap \partial h({\bar p})$,
    and hence ${\bar p}$ is a critical point of $f$ in problem~\eqref{Pr:DCproblem}.
\end{proof}

\begin{proposition}
    Let $(p^{(k)})_{k\in \mathbb N}$ be generated by \cref{Alg:DCA2}.
    Then, for all $N\in \mathbb{N},$ there holds
    \begin{equation*}
        \min_{k=0,1,\ldots,N}   d(p^{(k)},p^{(k+1)})
        \leq \biggl( \frac{ 2(f (p^{0})-f _{\inf})}{(N+1)\sigma} \biggr)^{1/2}.
    \end{equation*}
\end{proposition}
\begin{proof}
    It follows from~\eqref{eq:dsc} that
    $d^{2}(p^{(k)},p^{(k+1)})\leq (2/\sigma) \bigr( f (p^{(k)})-f (p^{(k+1)}) \bigr)$,
    for all $k\in \mathbb{N}$. Thus,
    \begin{equation*}
        (N+1)\min _{k=0,1,\ldots,N}\biggl( d^{2}(p^{(k)},p^{(k+1)}) \biggr)
        \leq \sum_{k=0}^{N}\frac{2}{\sigma} \biggl( f (p^{(k)})-f (p^{(k+1)}) \biggr)
        \leq \frac{2}{\sigma}\biggl( f (p^{0})-f _{\inf}\biggr),
    \end{equation*}
    where $f_{\inf}>-\infty$ is given by assumption~\ref{it:A2}.
    Therefore, the desired inequality directly follows.
\end{proof}

The last result of this section establishes a primal-dual asymptotic convergence of the
sequences generated by the DCA.
This result extends the known result from the Euclidean case,
cf.~\cite[Theorem 3]{TaoSouad1988}, to Hadamard manifolds.
Due to the nature of the problem, we will use the formulation of the DCA given
in \cref{Alg:DCA1}.

\begin{theorem}\label{th:pdac}
    Let $(p^{(k)})_{k\in \mathbb{N}}$ and $(X^{(k)})_{k\in \mathbb{N}}$ be
    the sequences generated by \cref{Alg:DCA1}. Then, the following statements hold:
    \begin{enumerate}
        \item
        \label{th:pdac_i}
            $g(p^{(k+1)})-h(p^{(k+1)})
            \leq h^{*}(p^{(k)},X^{(k)})-g^{*}(p^{(k)},X^{(k)})
            \leq g(p^{(k)})-h(p^{(k)})$,
            for all $k=0,1, \ldots$.
    \item
    \label{th:pdac_ii}
        $\displaystyle\lim_{k\to+\infty} ( g(p^{(k)})-h(p^{(k)}) )
        = \displaystyle\lim_{k\to+\infty} (h^{*}(p^{(k)},X^{(k)})-g^{*}(p^{(k)},X^{(k)}))
        = {\bar f} \geq f _{\inf}$.
    \item
    \label{th:pdac_iii}
        If the sequence $(p^{(k)})_{k\in \mathbb{N}}$ is bounded and ${\bar p}$
        is a cluster point of $(p^{(k)})_{k\in \mathbb{N}}$, then ${\bar p}\in \dom(g)$
        and there exists a cluster point ${\bar X}$ of $(X^{(k)})_{k\in \mathbb{N}}$
        such that
        \begin{subequations}
            \begin{equation}\label{pdac_eq1}
                \partial g({\bar p})\cap \partial h({\bar p})\neq \varnothing ,
            \end{equation}
            \begin{equation}\label{pdac_eq2}
                \lim_{k\to+\infty} ( h(p^{(k)})+h^{*}(p^{(k)},X^{(k)}) )
                    = h({\bar p})+h^{*}({\bar p},{\bar X})=0,
            \end{equation}
            \begin{equation}\label{pdac_eq3}
                \lim_{k\to+\infty} ( g(p^{(k)})+g^{*}(p^{(k)},X^{(k)}) )
                    = g({\bar p})+g^{*}({\bar p},{\bar X})=0.
            \end{equation}
            \begin{equation}\label{pdac_eq4}
                \partial_2 h^{*}({\bar p},{\bar X})\cap \partial_2 g^{*}({\bar p},{\bar X})
                    \neq \varnothing,
            \end{equation}
            \begin{equation}\label{pdac_eq5}
                g({\bar p})-h({\bar p})
                    = h^{*}({\bar p},{\bar X})-g^{*}({\bar p},{\bar X})
                    ={\bar f},
            \end{equation}
        \end{subequations}
    \end{enumerate}
\end{theorem}
\begin{proof}\ \\[-2\baselineskip]
    \begin{enumerate}
        \item By applying \cref{th:FYIN} with $q=p^{(k+1)}$, $p=p^{(k)}$, $X=X^{(k)}$, and $f=h$,
        we obtain that
        $h(p^{(k+1)})+h^{*}(p^{(k)},X^{(k)})
            \geq \langle X^{(k)},\exp^{-1}_{p^{(k)}}p^{(k+1)}\rangle$.
        Since \eqref{eq:DCAS_natural} implies that
        $\exp^{-1}_{p^{(k)}}p^{(k+1)}\in \partial_2 g^{*}(p^{(k)},X^{(k)})$,
        we can apply \cref{th:FYINPC} with $f=g$, $p=p^{(k)}$, $Y=\exp^{-1}_{p^{(k)}}p^{(k+1)}$, and
        $X=X^{(k)}$ to obtain
        $\langle X^{(k)}, \exp^{-1}_{p^{(k)}}p^{(k+1)} \rangle=g(p^{(k+1)})+g^{*}(p^{(k)},X^{(k)})$.
        Hence, we have $h(p^{(k+1)})+h^{*}(p^{(k)},X^{(k)}) \geq g(p^{(k+1)})+g^{*}(p^{(k)},X^{(k)})$,
        which is equivalent to the first inequality of \cref{th:pdac_i}.
        To prove the second one, we first note that since $X^{(k)}\in \partial h(p^{(k)})$
        and $h$ is convex, by using \cref{th:eqsubdifcf},
        we have $h^{*}(p^{(k)},X^{(k)})+h(p^{(k)})=0$.
        Thus, applying \cref{th:FYIN} with $q=p=p^{(k)}$, $X=X^{(k)}$ and $f=g$, we have
        $0\leq g^{*}(p^{(k)},X^{(k)})+g(p^{(k)})$, which combined with the last equality yields
        the second inequality of \cref{th:pdac_i}.
        \item First we recall that $f = g-h$ satisfies assumption~\ref{it:A2}.
        Thus, \cref{th:pdac_i} implies that $(f (p^{(k)}))_{k\in \mathbb{N}}$ is
        non-increasing and convergent. Hence
        $\lim_{k\to+\infty}(g(p^{(k)})-h(p^{(k)}))\eqqcolon{\bar f}\in {\mathbb R}$.
        Moreover, by using again \cref{th:pdac_i}, we also have
        \begin{equation*}
            \lim _{k\to+\infty}
                (h^{*}(p^{(k)},X^{(k)})-g^{*}(p^{(k)},X^{(k)}))
                \eqqcolon {\bar f} \in {\mathbb R}.
        \end{equation*}
        Finally, the inequality in \cref{th:pdac_ii} follows from assumption~\ref{it:A2}.
        \item To prove the first part, we assume that $(p^{(k)})_{k\in \mathbb{N}}$ is bounded
        and ${\bar p}$ a cluster point of $(p^{(k)})_{k\in \mathbb{N}}$.
        By using \cref{pr:assympc}, we conclude that ${\bar p}\in \dom (g)$ and that there
        exists a cluster point ${\bar X}$ of $(X^{(k)})_{k\in \mathbb{N}}$, such that
        ${\bar X}\in \partial g({\bar p})\cap \partial h({\bar p})$.
        Therefore, \eqref{pdac_eq1} is proved.
        Before proceeding with the proof we note that due to ${\bar p}\in \dom (g)$,
        assumption \ref{it:A3} implies that ${\bar p}\in \dom(h)$.

        To prove \eqref{pdac_eq2} note that since $X^{(k)}\in \partial h(p^{(k)})$,
        for all $k\in \mathbb{N}$, and $h$ is convex, from \cref{th:eqsubdifcf}, we have
        $h(p^{(k)})+h^{*}(p^{(k)},X^{(k)}) = 0$, for all $k \in \mathbb{N}$.
        Consequently,
        $\displaystyle\lim _{k\to+\infty} (h(p^{(k)})+h^{*}(p^{(k)},X^{(k)}))= 0.$
        Since ${\bar X}\in \partial h({\bar p})$, using again \cref{th:eqsubdifcf}, we have
        $h({\bar p})+h^{*}({\bar p},{\bar X}) =0$ and~\eqref{pdac_eq2} follows directly.

        To prove~\eqref{pdac_eq3} we first note that
        \begin{multline*}
            g(p^{(k)})+g^{*}(p^{(k)},X^{(k)})
            = g(p^{(k)})- h(p^{(k)})- \big(h^{*}(p^{(k)},X^{(k)})\\-g^{*}(p^{(k)},X^{(k)})\big)
                + h(p^{(k)})+ h^{*}(p^{(k)},X^{(k)}).
        \end{multline*}
        Thus, using \cref{th:pdac_ii} together with~\eqref{pdac_eq2}, we have
        $\displaystyle\lim_{k\to+\infty} (g(p^{(k)})+g^{*}(p^{(k)},X^{(k)}))=0$.
        Since ${\bar X}\in \partial g({\bar p})$, using again \cref{th:eqsubdifcf},
        we have $g({\bar p})+g^{*}({\bar p},{\bar X}) =0$, which combined with the last
        equality yields~\eqref{pdac_eq3}.

        We proceed to prove~\eqref{pdac_eq4}. For that, we assume without loss of generality
        that $\lim _{k\to+\infty} p^{(k)}={\bar p}$.
        Now, by applying \cref{th:FYIN} with $f=h$, $p={\bar p}$, $q= p^{(k)}$, we obtain
        \begin{equation*}
            h(p^{(k)})+h^{*}({\bar p},Y)
            \geq \langle Y,\exp^{-1}_{\bar p}p^{(k)} \rangle,
            \qquad \text{ for all } Y \in T_{{\bar p}}\mathcal M,
            \quad\text{ and all } k\in {\mathbb N}.
        \end{equation*}
        Thus, by using \cref{def:linf}, $\displaystyle\lim_{k\to+\infty} p^{(k)}={\bar p}$,
        \cref{pr:cont_pi}, \cref{pr:cont_pi_i} and \cref{pr:cont_pi_iii},
        and that $h$ is lsc, we have
        $h({\bar p})+h^{*}({\bar p},Y)
            =\displaystyle\liminf_{k\to+\infty}h(p^{(k)})+h^{*}({\bar p},Y) \geq 0$.
        Thus, the second equality in \eqref{pdac_eq2} implies that
        \begin{equation*}
            h^{*}({\bar p},Y) \geq h^{*}({\bar p},X),
                \qquad \text{ for all } Y \in T_{{\bar p}}\mathcal M.
        \end{equation*}
        Hence, $0\in \partial_2h^{*}({\bar p},X)$. Similarly, by using \eqref{pdac_eq3},
        we can also show that $0\in \partial_2g^{*}({\bar p},X)$.
        Therefore, $0\in \partial_2h^{*}({\bar p},X) \cap \partial_2g^{*}({\bar p},X)$,
        which proves \eqref{pdac_eq4}.

        Finally, we prove \eqref{pdac_eq5}.
        Combining the second equality in~\eqref{pdac_eq2} and~\eqref{pdac_eq3},
        we obtain the first equality in~\eqref{pdac_eq5}.
        To prove the second inequality, we first note that
        ${\bar p}\in \dom (g)\subset \mbox{\emph{int dom}}(h)$.
        Since $h$ is convex, it is continuous in $\mbox{\emph{int dom}}(h)$,
        which implies that $\lim _{k\to+\infty} h(p^{(k)})=h({\bar p})$.
        Thus, using \cref{th:pdac_ii}, we conclude that
        \begin{equation*}
            \lim_{k\to+\infty}g(p^{(k)})
            = \lim_{k\to+\infty}( g(p^{(k)})-h(p^{(k)})) + \lim _{k\to+\infty}h(p^{(k)})
            ={\bar f}+ h({\bar p}).
        \end{equation*}
        Hence, using \cref{def:linf}, we hav
        $\lim _{k\to+\infty} g(p^{(k)})=\liminf_{k\to+\infty} g(p^{(k)})= g({\bar p})$.
        Therefore, we obtain that $g({\bar p})- h({\bar p})=\bar f$,
        which concludes the proof.\qedhere
    \end{enumerate}
\end{proof}

\section{Examples}\label{Sec:Examples}
In this section we consider examples of DC functions on the Hadamard manifold of symmetric positive definite matrices. These examples can also be seen as constraint problems on the Euclidean space of square matrices, but they are not DC problems thereon. Only by imposing the manifold structure on the constrained set, namely the symmetric positive definite matrices set, both components of the problem become convex.

Formerly, we consider the symmetric positive definite (SPD) matrices cone ${\mathbb P}^{n}_{++}$.
Following~\cite{Rothaus1960}, see also \cite[Section 6.3]{NesterovTodd2002}, we introduce the Hadamard manifold,
\begin{equation} \label{eq:RiemPmatrix}
    \mathcal{M}\coloneqq ({\mathbb P}^n_{++}, \langle \cdot , \cdot \rangle)
\end{equation}
endowed with the Riemannian metric given by
\begin{equation} \label{eq:metric}
    \langle X,Y \rangle_p \coloneqq \operatorname{tr}(Xp^{-1}Yp^{-1}),
\end{equation}
for $p\in \mathcal{M}$ and $X,Y\in T_p\mathcal{M}$, where $\mbox{tr}(p)$ denotes the trace of the matrix $p\in {\mathbb P}^{n}_{++}$, $T_p\mathcal{M}\approx\mathbb{P}^n$ is the tangent space of $\mathcal{M}$ at $p$ and ${\mathbb P}^{n}$ denotes the set of symmetric matrices of order $n\times n$.
Further details about the Hadamard manifold $\mathcal{M}$ can be found, for example, in~\cite[Theorem~1.2. p. 325]{Lang1999}.
The \emph{exponential map and its inverse} at a point $p\in \mathcal{M}$ are given, respectively,  by
\begin{align}
    \exp_{p}X       & \coloneqq p^{1/2} e^{p^{-1/2}Xp^{-1/2}} p^{1/2}, \qquad X\in T_p\mathcal{M}, p \in \mathcal M \label{eq:exp}\\
    \exp^{-1}_{p}{q}& \coloneqq p^{1/2} \log({p^{-1/2}qp^{-1/2}}) p^{1/2}, \qquad p, q\in \mathcal{M}  \label{eq:invexp}.
\end{align}
The \emph{dimension} of the manifold is given by $\operatorname{dim}_{\mathbb P_{++}^n} = \frac{n(n+1)}{2}$.

The \emph{gradient} of a differentiable function  $f\colon{\mathbb P}^n_{++}\to {\mathbb R}$ is  given by
\begin{equation}
\label{eq:Grad}
\grad f(p)=pf'(p)p.
\end{equation}
where  $f'(p)$ is the Euclidean gradient of $f$ at $p$. If $f$ is twice differentiable, then the \emph{hessian} of $f$ is given by
\begin{equation}
\label{eq:Hess}
\operatorname{Hess}\,f(p)X=pf''(p)Xp+\frac{1}{2}\bigl[  Xf'(p)p+  pf'(p)X \bigr],
\end{equation}
where $X\in T_p\mathcal{M}$ and  $f''(p)$ is  the  Euclidean hessian of  $f$ at $p$.

In general, subproblem~\eqref{eq:DCAS} in \cref{Alg:DCA2} is not convex; nevertheless, in some special cases, as illustrated by the following examples, it actually is convex.
To begin, recall that the gradient and hessian of a function $p\mapsto \varphi(\det(p))$, where $\varphi\colon\mathbb{R}_{++}\to \mathbb{R}$ is  twice differentiable, is given by
\begin{align}
\grad\varphi(\det(p))&=\bigl(\varphi'(\det(p))\det(p)\bigr)p, \label{eq:Graddetf}\\
\operatorname{Hess}\,\varphi(\det(p))v&=\bigl(\varphi''(\det(p))(\det(p))^2+\varphi'(\det(p))\det(p)\bigr)tr(p^{-1}v)p\label{eq:Hessdetf},
\end{align}
where $v\in T_p\mathcal{M}$, $\varphi'$ and $\varphi''$ are the first and second derivative of $\varphi$, respectively.
\begin{example} \label{ex:loglog}
    Consider the following  optimization  problem
    \begin{equation} \label{eq:loglog}
    \argmin_{p\in \mathcal M} f(p), \qquad \text{where }  f(p)\coloneqq \varphi_1(\det(p))-\varphi_2(det(p)),
    \end{equation}
    where the function $\varphi_i\colon\mathbb{R}_{++}\to \mathbb{R}$ are twice differentiable satisfying  $\varphi_i''(t)t^2+\varphi_i'(t)t\geq 0$, for all $t\in \mathbb{R}_{++}$ and $i=1,2$.
    Indeed, by using~\eqref{eq:Grad} and~\eqref{eq:Hess}, we can show that~\eqref{eq:loglog} is a DC problem with components
    \begin{equation} \label{eq:loglog1}
    g(p)=\varphi_1(\det(p)), \qquad \quad h(p)=\varphi_2(\det(p)).
    \end{equation}
    This follows from~\eqref{eq:Hessdetf} and $\varphi_i''(t)t^2+\varphi_i'(t)t\geq 0$, for all $t\in \mathbb{R}_{++}$ that $\langle \operatorname{Hess}\,\varphi_i(\det(p))X, X\rangle \geq 0$, for all $X\in T_p\mathcal{M}$ and $i=1,2$, which implies that $g$ and $h$ are convex.

    By using \eqref{eq:Graddetf} we conclude that critical points of $f$ are  matrices ${\bar p}\in \mathbb{P}_{++}^n$ such that
    \begin{equation} \label{eq:DCCriticalPointsG}
        \varphi'_1(\det({\bar p}))=\varphi'_2(\det({\bar p})).
    \end{equation}
    Now, considering that  $h$ is a differentiable function, consider the subproblem associated to the problem~\eqref{eq:loglog}
    \begin{equation} \label{eq:diffcasedetf}
        \argmin_{p\in \mathcal M}  \psi(p),   \qquad  \text{where } \psi(p) \coloneqq \varphi_1(\det(p))-\big\langle  \grad(\varphi_2(\det(q))), \exp^{-1}_{q}{p} \big\rangle
    \end{equation}
    It is worth noting that if we use \cref{Alg:DCA2} to solve problem~\eqref{eq:loglog}, the subproblem~\eqref{eq:DCAS} to be addressed has the form~\eqref{eq:diffcasedetf}.
    In general, subproblem~\eqref{eq:DCAS} is not convex; nevertheless, we will show now that~\eqref{eq:diffcasedetf} is a convex problem.
    In fact, by using second equality in~\eqref{eq:Graddetf} it follows from~\eqref{eq:diffcasedetf} that
    \begin{equation} \label{eq:diffcasedetf1}
        \psi(p) =\varphi_1(\det(p))-\bigl(\varphi'_2(\det(q))\det(q)\bigr) \big\langle  q, \exp^{-1}_{q}{p} \big\rangle.
    \end{equation}
    On the other hand, by using the exponential in~\eqref{eq:invexp} and the metric in~\eqref{eq:metric} we obtain that
    \begin{equation*}
        \big\langle q, \exp^{-1}_{q}{p} \big\rangle= \mbox{tr}\big( \log({q^{-1/2}pq^{-1/2}}) \big).
    \end{equation*}
    Since $ \operatorname{tr}\log Z =\log \det Z$, for any matrix $Z$, the last equality becomes
    \begin{equation} \label{eq:cipmdetf}
        \big\langle  q, \exp^{-1}_{q}{p} \big\rangle= \log \det (p) - \log \det (q).
    \end{equation}
    Combining~\eqref{eq:diffcasedetf1} with~\eqref{eq:cipmdetf}, the function $\psi$ in subproblem \eqref{eq:diffcasedetf} is rewritten equivalently as
    \begin{equation} \label{eq:diffcasedetffd}
        \psi(p) =\varphi_1(\det(p))-\bigl(\varphi'_2(\det(q))\det(q)\bigr) \bigl(\log \det (p) - \log \det (q)\bigr).
    \end{equation}
    Since the matrix $q\in \mathcal{M}$  is fixed and the function $g(p)=\varphi_1(\det(p))$  is convex, proving that $\psi$ is convex is sufficient to prove that the function $\Upsilon(p)=-\log \det(p)$ is convex.
    Applying \eqref{eq:Hessdetf} with $\varphi=\log$ we conclude that $\operatorname{Hess}\,\Upsilon(p)=0$, for all $p$,  which implies that $\Upsilon$ is convex.
    In conclusion, the objective function $f$ in problem~\eqref{eq:loglog} is not convex in general, while the function $\psi$ in the associated subproblem \eqref{eq:diffcasedetf} is.
    Let us conclude by presenting some  functions  $\varphi\colon \mathbb{R}_{++}\to \mathbb{R}$ satisfying the condition $\varphi''(t)t^2+\varphi'(t)t\geq~0$, for all $t\in \mathbb{R}_{++}$:
    \begin{enumerate}
        \item $\varphi_1(t)=a_1(\log(t))^{2(b+1)}$ and $\varphi_2(t)=a_2(\log(t))^{2b}$  with  $a_1, a_2 \in  \mathbb{R}_{++}$ and  $b\geq 1$. \label{ex:loglog:concrete1}
        \item $\varphi_1(t)={\bar a}\log(t^b+c_1)-{\hat a}\log(t)$  and  $\varphi_2(t)=\log(t+c_2)$  with  ${\bar a}, {\hat a}, b, c_1, c_2 \in  \mathbb{R}_{++}$. Note that,   if $ab>d+1$,  then $\varphi_1-\varphi_2$ has a critical point.
        \item $\varphi_1(t)=a_1t^{b_1+2}$ and $\varphi_2(t)=a_2t^{b_2+2}$ with  $a_1, a_2,   b_1, b_2 \in  \mathbb{R}_{+}$.
    \end{enumerate}
    Finally, it is worth noting that these functions $g$ and $h$ in~\eqref{eq:loglog1} associated with these  problems are in general not Euclidean convex functions.
    Consequently, \eqref{eq:loglog} is not a Euclidean DC problem. As we just derived, they are DC in the Hadamard  manifold~\eqref{eq:RiemPmatrix}.
\end{example}
Let us examine at another set of examples that are not DC Euclidean problems but are DC in the Hadamard manifold \eqref{eq:RiemPmatrix} described above.
\begin{example}
Consider the following optimization problem
\begin{equation} \label{eq:trlog}
\argmin_{p\in \mathcal M} f(p), \qquad \text{where } f(p)\coloneqq \varphi_1(\mbox{tr}(p))-\varphi_2(det(p)),
\end{equation}
where the function $\varphi_i\colon\mathbb{R}_{++}\to \mathbb{R}$ are twice differentiable satisfying  the following conditions
\begin{equation} \label{eq:condtrlog}
\varphi_1'(t)\geq 0, \quad \varphi_1''(t)\geq 0, \quad  \varphi_2''(t)t^2+ \varphi_2(t)t\geq 0 \qquad \forall t\in \mathbb{R}_{++}.
\end{equation}
In general, the objective function $f$ in the  problem~\eqref{eq:trlog} is not convex in either the Euclidean context nor the Hadamard manifold~\eqref{eq:RiemPmatrix}.
However, by using~\eqref{eq:Grad} and~\eqref{eq:Hess},   we prove that the components in~\eqref{eq:trlog}, denoted by
\begin{equation} \label{eq:trlog1}
    g(p)=\varphi_1(\mbox{tr}(p)), \qquad\text{ and }\qquad h(p)=\varphi_2(\det(p)),
\end{equation}
  which, in general, are not convex Euclidean, are convex functions on the Hadamard manifold~\eqref{eq:RiemPmatrix}  since conditions in~\eqref{eq:condtrlog} hold.
  Therefore, \eqref{eq:trlog} is a DC optimization  problem.
  In addition, by using \eqref{eq:Grad}, we can show that the gradients of $g$ and $h$ are given by
  \begin{equation} \label{eq:gradghmG}
    \grad  g(p)= \varphi_1'(\mbox{tr}(p)) p^2, \qquad \qquad  \grad  h(p)= \bigl( \varphi_2'(\det(p)) \det(p)\bigr)p,
  \end{equation}
  respectively.
  By using \eqref{eq:gradghmG} we conclude that critical points of $f$ are matrices ${\bar p}\in \mathbb{P}_{++}^n$ such that
\begin{equation} \label{eq:DCCriticalPointtrdet}
 \varphi_1'(\mbox{tr}({\bar p})){\bar p}=\bigl( \varphi_2'(\det({\bar p})) \det {\bar p}\bigr)I.
\end{equation}
Using the same arguments as in \cref{ex:loglog}, we can show that the subproblem associated with problem~\eqref{eq:trlog} is given by
\begin{equation} \label{eq:Subtrlog}
  \argmin_{p\in \mathcal M}  \psi(p),   \quad  \text{where }   \psi(p) =\varphi_1(\mbox{tr}(p))-\bigl(\varphi'_2(\det(q))\det(q)\bigr) \bigl(\log \det (p) - \log \det (q)\bigr),
\end{equation}
for a fixed $q\in \mathcal{M}$, and  the objective function $\psi$  is convex.
Finally, let us present some functions satisfying the condition~\eqref{eq:condtrlog}.
\begin{enumerate}
\item $\varphi_1(t)=a_1t^{b_1}$  and  $\varphi_2(t)=a_2t^{b_2}$ with $a_1\geq1$ and  $a_2, b_1, b_2>0$ such that $a_1b_1n^{b_1-1}=a_2b_2$.
\item $\varphi_1(t)=ae^{bt}$  and  $\varphi_2(t)=\frac{1}{2}abe^{nb}t^2$, with $a, b>0$.
\end{enumerate}
\end{example}
\section{Numerics}\label{Sec:Numerics}

In this section, we present several numerical examples. On the one hand, we compare the algorithm to two existing algorithms and, on the other hand, illustrate in a third example how optimization problems can be reformulated into DC problems to use this structure as an advantage in numerical computations. For all numerical examples, the \cref{Alg:DCA2} is implemented in
Julia 1.8.5~\cite{BezansonEdelmanKarpinskiShah2017} within the package
\lstinline!Manopt.jl!~\cite{Bergmann2022}
version 0.4.12, using a trust region solver to solve the optimization problem
in~\eqref{eq:DCAS} within every step,
including a generic implementation of the corresponding cost and gradient.
This way, the algorithm is easy-to-use, while when a more efficient computation for either
cost and gradient of the sub problem or even a closed form solution is available, they
can benefit to speed up the computation, when provided.
Together with \lstinline!Manifolds.jl!~\cite{AxenBaranBergmannRzecki2021}
this algorithm can be used on arbitrary manifolds.
All times refer to running the experiments on an Apple MacBook Pro M1 (2021), 16 GB Ram, Mac OS Ventura 13.0.1.

\subsection{A comparison to the Difference of Convex Proximal Point Algorithm}

We first consider the problem
\begin{equation*}
    \argmin_{p\in\mathcal M}  \bigl( \log\bigr(\det(p)\bigr)\bigr)^4 - \bigl(\log \det(p) \bigr)^2.
\end{equation*}
on $\mathcal P_{++}^n$.   Here we have $f(p) = g(p) - h(p)$ where
$g(p) = \varphi_1\bigl(\det(p)\bigr)$, $\varphi_1(t) = (\log t)^4$,
and $h(p) = \varphi_2\bigl(\det(p)\bigr)$, $\varphi_2(t) = (\log t)^2$,
which fits \cref{ex:loglog}, \cref{ex:loglog:concrete1}.  The critical points of this problem are the matrices $p^*\in \mathcal P_{++}^n$ such that $\det(p^*)=e^{1/\sqrt{2}}$. We have $f(p^*) = -\frac{1}{4}$,  for each critical point $p^*$.

We compare the DCA with the Difference of Convex Proximal Point Algorithm (DCPPA)
as introduced in~\cite{SouzaOliveira2015}. The algorithm is also available in \lstinline!Manopt.jl!,
implemented in the same generic manner, as the DCA explained above.
This means that the proximal map can be considered as a subproblem to solve.
When only $g$ and its gradient $\operatorname{grad} g$ are provided, the subproblem
is generated in a generic manner, that is, a default implementation of the minimization problem
that corresponds to step 3 of the DCPPA-Algorithm from~\cite{SouzaOliveira2015} is generated.
This is also the scenario we use for our example.
For the case that $\operatorname{prox}_{\lambda g}(p)$ is available, e.g.\ in closed form, it can be provided to the algorithm for speed-up.

For both DCA and DCPPA, the generation of the generic subproblem is the default in
\lstinline!Manopt.jl! as soon as the gradient $\operatorname{grad} g$ of $g$ is provided.
The function calls look like

\begin{lstlisting}
difference_of_convex_algorithm(M, f, g, grad_h, p0; grad_g=grad_g)
difference_of_convex_proximal_point(M, grad_h, p0; g=g, grad_g=grad_g)
\end{lstlisting}

By default, further an approximation the Hessian of both sub problems by a Riemannian variant
of forward difference from the gradient is used.
This enables the use of the~\lstinline!trust_regions!\footnote{see \href{https://manoptjl.org/stable/solvers/trust_regions/}{manoptjl.org/stable/solvers/trust\_regions/} for details}
algorithm to solve the sub-problem.
To make both algorithms comparable we
\begin{itemize}
    \item for both sub solvers we \emph{stop} when the gradient norm (of the subproblem's gradient) is below $10^{-10}$ or after $5000$ iterations
    if the gradient does not get small.
    \item for both algorithms DCA and DCPPA when the gradient norm (of $f$) is below $10^{-10}$. We also have a fall back to stop after $100$ iterations if the gradient norm is not hit.
    \item the proximal parameter in the DCPPA to a constant of $\lambda=\frac{1}{2n}$
    \item for both algorithms we set $p^{(0)} = \log(n) I_n$ as the initial point, where $I_n$ denotes the $n\times n$ identity matrix
\end{itemize}

For the matrix size $n$ of $\mathbb P_{++}^n$ we set $n=2,3,\ldots,80$ to compare the algorithm
for different manifold sizes, which yields manifolds of dimension $d=\frac{n(n+1)}{2}$.

In \cref{fig:DCAvsDCPPA-time} we compare the different run times for both the DCA and DCPPA.
These were obtained using the \lstinline!@benchmark! macro from \lstinline!BenchmarkTools.jl!~\cite{ChenRevels2016},

Up to a dimension of approximately $d=40$ (or $8\times 8$ spd.\ matrices) the DCA is faster.
This includes the important case of $3\times 3$ spd.\ matrices, that is one representation of diffusion tensors,
where the DCA takes only $5.2434\cdot10^{-3}$ seconds while the DCPPA takes
$2.2672\cdot10^{-2}$ seconds.
For higher-dimensional problems, cf. \cref{fig:DCAvsDCPPA-time:high},
the DCPPA seems to only increase very slowly, where $d=465$, or $30\times 30$ spd.\ matrices,
seems to be an outlier, where DCPPA takes over $22$ seconds, while otherwise it
stays around about half a second, even for the last case shown,
i.e.\ $d=99$ (or $44\times 44$ spd.\ matrices).

\begin{figure}[tbp]\centering
    \begin{subfigure}{.49\textwidth}
        \begin{tikzpicture}
            \begin{axis}[
                ymin=0,
                ymax=0.15,
                xmin=0,
                xmax=100,
                legend style={draw=none, legend columns=-1, at={(0.5,1.0)},anchor=north},
                width=.9\textwidth,
                height=.2\textheight,
                axis lines=left,
                ylabel near ticks,
                xlabel near ticks,
                xlabel={manifold dimension $d=\operatorname{dim}_{\mathbb P_{++}^n}$},
                ylabel={run time (sec.)},
            ]
                \addplot[color=TolMutedIndigo, thick]
                    table [x=d, y=dcat, col sep=comma] {data/DCAvsDCPPA-80-benchmark-shortened.csv};
                \addplot[color=TolMutedTeal, thick]
                    table [x=d, y=dcppat, col sep=comma] {data/DCAvsDCPPA-80-benchmark-shortened.csv};
                \legend{DCA, DCPPA}
            \end{axis}
        \end{tikzpicture}
        \caption{Low dimensional performance}
        \label{fig:DCAvsDCPPA-time:low}
    \end{subfigure}
    \begin{subfigure}{.49\textwidth}
        \begin{tikzpicture}
            \begin{axis}[
                ymin=0,
                ymax=13,
                xmin=0,
                xmax=1000,
                legend style={draw=none, legend columns=-1, at={(0.5,1.0)},anchor=north},
                width=.9\textwidth,
                height=.2\textheight,
                axis lines=left,
                ylabel near ticks,
                xlabel near ticks,
                xlabel={manifold dimension $d=\operatorname{dim}_{\mathbb P_{++}^n}$},
                ylabel={run time (sec.)},
            ]
                \addplot[color=TolMutedIndigo, thick]
                    table [x=d, y=dcat, col sep=comma] {data/DCAvsDCPPA-80-benchmark.csv};
                \addplot[color=TolMutedTeal, thick]
                    table [x=d, y=dcppat, col sep=comma] {data/DCAvsDCPPA-80-benchmark.csv};
                \legend{DCA, DCPPA}
            \end{axis}
        \end{tikzpicture}
        \caption{High dimensional performance}
        \label{fig:DCAvsDCPPA-time:high}
    \end{subfigure}
    \caption{A comparison of the run times of DCA and DCPPA for different manifold dimensions.}
    \label{fig:DCAvsDCPPA-time}
\end{figure}
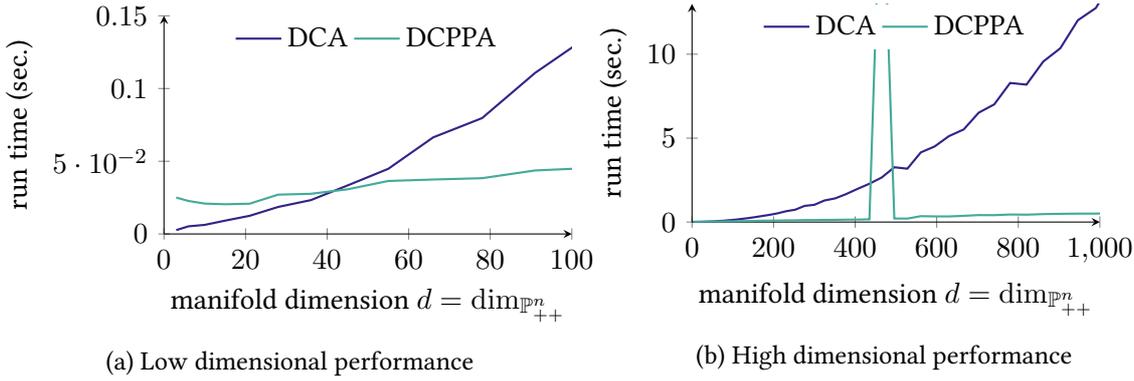

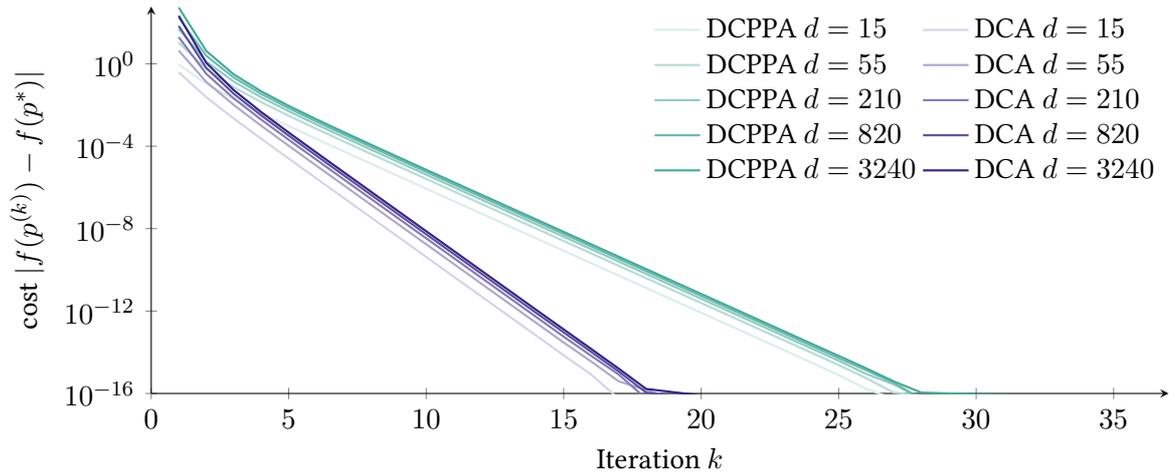
\begin{figure}[tbp]\centering
        \begin{tikzpicture}
            \begin{axis}[
                ymin=1e-16,
                ymax=6*1e2,
                ymode=log,
                xmin=0,
                xmax=37,
                legend style={draw=none, transpose legend, legend columns=5, at={(1.0,1.0)},anchor=north east},
                legend cell align={left},
                width=.95\textwidth,
                height=.3\textheight,
                axis lines=left,
                ylabel near ticks,
                xlabel near ticks,
                xlabel={Iteration $k$},
                ylabel={cost $\lvert f(p^{(k)})-f(p^*) \rvert$},
            ]
            \addplot[color=TolMutedTeal!20, thick]
                table [x=i, y=fabs, col sep=comma] {data/DCAvsDCPPA-80-DCPPA-5.csv};
            \addplot[color=TolMutedTeal!40, thick]
                table [x=i, y=fabs, col sep=comma] {data/DCAvsDCPPA-80-DCPPA-10.csv};
            \addplot[color=TolMutedTeal!60, thick]
                table [x=i, y=fabs, col sep=comma] {data/DCAvsDCPPA-80-DCPPA-20.csv};
            \addplot[color=TolMutedTeal!80, thick]
                table [x=i, y=fabs, col sep=comma] {data/DCAvsDCPPA-80-DCPPA-40.csv};
            \addplot[color=TolMutedTeal, thick]
                table [x=i, y=fabs, col sep=comma] {data/DCAvsDCPPA-80-DCPPA-80.csv};
            \addplot[color=TolMutedIndigo!20, thick]
                table [x=i, y=fabs, col sep=comma] {data/DCAvsDCPPA-80-DCA-5.csv};
            \addplot[color=TolMutedIndigo!40, thick]
                table [x=i, y=fabs, col sep=comma] {data/DCAvsDCPPA-80-DCA-10.csv};
            \addplot[color=TolMutedIndigo!60, thick]
                table [x=i, y=fabs, col sep=comma] {data/DCAvsDCPPA-80-DCA-20.csv};
            \addplot[color=TolMutedIndigo!80, thick]
                table [x=i, y=fabs, col sep=comma] {data/DCAvsDCPPA-80-DCA-40.csv};
            \addplot[color=TolMutedIndigo, thick]
                table [x=i, y=fabs, col sep=comma] {data/DCAvsDCPPA-80-DCA-80.csv};
            \legend{DCPPA $d=15$,
                    DCPPA $d=55$,
                    DCPPA $d=210$,
                    DCPPA $d=820$,
                    DCPPA $d=3240$,
                    DCA $d=15$,
                    DCA $d=55$,
                    DCA $d=210$,
                    DCA $d=820$,
                    DCA $d=3240$,
                }
            \end{axis}
        \end{tikzpicture}
    \caption{A comparison of how close the cost function is to the actual minimum for different sizes of problems and both algorithms.}
    \label{fig:DCAvsDCPPA-cost}
\end{figure}

Comparing the number of iterations, we observe that after the first 5 experiments,
so starting from a dimension of $21$ ($6\times 6$ spd. matrices), the number of
iterations stabilizes around $25$ iterations for the DCA and $38$ for the DCPPA.

We compare different developments of the cost function in \cref{fig:DCAvsDCPPA-cost}.
Since for all dimensions we know that $f(p^*) = -\frac{1}{4}$ we plot $\lvert f(p^{(k)}) - f(p^*) \rvert$ over the iterations for the manifold
dimensions $d=15, 55, 210, 820, 3240$, that is the $n\times n$ matrices for $n=5,10,20,40,80$.
The initial value $p^{(0)}$ was chosen as above, which yields that the value $f(p^{(1)})$
is always below $10^{3}$ in our experiments.
All these different dimensions show the same slope in the decrease of the cost function $f(p)$
for both the DCA as well as the DCPPA.
While for DCPPA the cost seems to be below
$10^{-16}$ close to the minimum for a few iterations already, before the stopping criterion of a gradient norm
$\rVert \operatorname{grad} f(p^{(k)})\lVert_{p^{(k)}} < 10^{-10}$ is reached.

The development of the cost function illustrates, that DCA converges faster than
DCPPA, such that the choice of the sub solver seems to be crucial for the run time,
which for these experiments we configures equally to compare the algorithms and
not sub solvers.

\subsection{The Rosenbrock Problem}
The Rosenbrock problem consists of
\begin{equation}\label{eq:Rosenbrock}
    \argmin_{x\in\mathbb R^2} a\bigl( x_{1}^2-x_{2}\bigr)^2 + \bigl(x_{1}-b\bigr)^2,
\end{equation}
where $a,b>0$ are positive numbers, classically $b=1$ and $a \gg b$, see \cite{Rosenbrock1960}.
Note that the function is non-convex on $\mathbb R^2$.
The minimizer $x^*$ is given by $x^* = (b,b^2)^\mathrm{T}$,
and also the (Euclidean) Gradient can be directly stated as
\begin{equation}\label{eq:Rosenbrock:gradf}
    \nabla f(x) =
    \begin{pmatrix}
        4a(x_1^2-x_2)\\
        -2a(x_1^2-x_2)
    \end{pmatrix}
    +
    \begin{pmatrix}
        2(x_1-b)\\
        0
    \end{pmatrix}
\end{equation}

We introduce a new metric for $\mathcal M = \mathbb R^2$:
For any $p \in \mathbb R^2$ we define
\begin{equation*}
    G_p \coloneqq
    \begin{pmatrix}
        1+4p_{1}^2 & -2p_{1} \\
        -2p_{1} & 1
\end{pmatrix}, \text{ which has the inverse matrix } G^{-1}_p =
\begin{pmatrix}
    1 & 2p_1\\
    2p_{1} & 1+4p_{1}^2 \\
\end{pmatrix}.
\end{equation*}
We define the inner product on $T_p\mathcal M = \mathbb R^2$ as
\begin{equation*}
    \langle X,Y \rangle_p = X^\mathrm{T}G_pY
\end{equation*}
In the following we refer to $\mathbb R^2$ with the default Euclidean metric
further as just $\mathbb R^2$ and to the same space with this new metric as $\mathcal M$.

The exponential and logarithmic map are given as
\begin{equation*}
    \exp_p(X) =
    \begin{pmatrix}
        p_1 + X_1\\
        p_2 + X_2 + X_1^2
    \end{pmatrix},\qquad
    \exp^{-1}_p(q) =
    \begin{pmatrix}
        q_1 - p_1\\
        q_2 - p_2 - (q_1-p_1)^2
    \end{pmatrix}.
\end{equation*}

Given some function $h\colon \mathcal M \to \mathbb R$, its Riemannian gradient
$\operatorname{grad} h\colon \mathcal M \to T\mathcal M$ can be computed from the
Euclidean one by
\begin{equation*}
    \operatorname{grad} h(p) = G_p^{-1}\nabla h(p).
\end{equation*}
Denoting the two components of the Euclidean gradient by $\nabla h(p) = (h'_1(p), h'_2(p))^\mathrm{T}$ we can derive that given two points $p,q\in\mathcal M$ we have
\begin{equation}\label{eq:gradh-log}
    \begin{split}
        \Bigl\langle
        \operatorname{grad} h(q),
        \exp_q^{-1}(p)
        \Bigr\rangle_q
        &=
         \bigl(\exp_q^{-1}(p)\bigr)^\mathrm{T}
        \nabla h(q)
        \\&
        = (p_1-q_1)h'_1(q) + (p_2 - q_2 - (p_1-q_1)^2)h'_2(q)
    \end{split}
\end{equation}

For the difference of convex algorithm we split the cost function from the Rosenbrock
problem \eqref{eq:Rosenbrock} as $f(x) = g(x) - h(x)$ with
\begin{equation*}
    g(x) = a\bigl( x_{1}^2-x_{2}\bigr)^2 + 2\bigl(x_{1}-b\bigr)^2
    \quad\text{ and }\quad
    h(x) = \bigl(x_{1}-b\bigr)^2.
\end{equation*}
Using the isometry $\psi\colon \mathbb R^2 \to \mathcal M, \mathbf{z} \mapsto (z_1, z_1^2-z_2)$
we get
\begin{equation*}
    (h\circ\psi)(x) = h(x_1, x_1^2-x_2) = (x_1-b)^2
\end{equation*}
and hence $h$ is (geodesically) convex on $\mathcal M$.

The corresponding Euclidean gradients of $g$ and $h$ are
\begin{equation*}
    \nabla g(p) =
    \begin{pmatrix}
        4a(x_1^2-x_2)\\
        -2a(x_1^2-x_2)
    \end{pmatrix}
    +
    \begin{pmatrix}
        4(x_1-b)\\
        0
    \end{pmatrix}
    \quad\text{ and }\quad
    \nabla h(p) =
    \begin{pmatrix}
        2(p_1-b)\\
        0
    \end{pmatrix},
\end{equation*}
So especially the second component $h_2'(p) = 0$.
\\
Considering the sub-problem \cref{eq:DCAS} from \cref{Alg:DCA2}, we obtain
together with \eqref{eq:gradh-log} for some fixed $q$ that
\begin{equation*}
    \begin{split}
        \varphi(p)
        &=
        g(p) -
        \Bigl\langle
        \operatorname{grad} h(q),
        \exp_q^{-1}(p)
        \Bigr\rangle_q
        \\
        &=
        a\bigl( p_{1}^2-p_{2}\bigr)^2
        + 2\bigl(p_{1}-b\bigr)^2 - 2(q_1-b)(p_1-q_1)
        \\
        &= a\bigl( p_{1}^2-p_{2}\bigr)^2
        + 2\bigl(p_{1}-b\bigr)^2 - 2(q_1-b)p_1 + 2(q_1-b)q_1,
    \end{split}
\end{equation*}
where the last term is constant with respect to $p$ and hence
irrelevant when determining a minimizer. The Euclidean Gradient reads
\begin{equation*}
    \nabla \varphi(p)
    = \begin{pmatrix}
        4a p_1(p_1^2-p_2) + 4(p_1-b) - 2(q_1-b)\\
        -2a(p_1^2-p_2)
    \end{pmatrix}
\end{equation*}
and the Riemannian gradient is similar to before $\operatorname{grad} \varphi(p)
= G_p^{-1}\nabla \varphi(p)$.
This allows to use for example the Riemannian gradient descent from \lstinline!Manopt.jl!
to be used as a sub-solver for \cref{eq:DCAS} within \cref{Alg:DCA2}.

Since also for the Rosenbrock function the Riemannian gradient can be easily computed in the same manner from \eqref{eq:Rosenbrock:gradf},
we can now compare three different first order methods:
\begin{enumerate}
    \item The Euclidean gradient descent algorithm on $\mathbb R^2$,
    \item The Euclidean Difference of Convex Algorithm on $\mathbb R^2$
    \item The Riemannian gradient descent algorithm on $\mathcal M$,
    \item The Riemannian Difference of Convex Algorithm on $\mathcal M$,
    using Riemannian gradient descent as a sub-solver
\end{enumerate}
All algorithms use \lstinline!ArmijoLinesearch(M)! when determining the step size in gradient descent,
and all stop either after 10 million steps, or when the change between two successive iterates drops below \lstinline!1e-16!.
The sub solver in the DCA is set to stop when the gradient is at \lstinline!1e-16! in norm or at 1000 iterations.

\begin{figure}[tbp]\centering
    \begin{tikzpicture}
        \begin{axis}[
            xmode=log,
            ymode=log,
            ymin=1e-16,
            ymax=3*1e4,
            xmin=0,
            xmax=8*1e7,
            legend style={draw=none, legend columns=-1, at={(0.5,1.1)},anchor=north},
            width=.9\textwidth,
            height=.25\textheight,
            axis lines=left,
            ylabel near ticks,
            xlabel near ticks,
            xlabel={Iteration $k$},
            ylabel={$f(p^{(k)})$},
        ]
            \addplot[color=TolMutedOlive, thick]
                table [x=i, y=y, col sep=comma] {data/Rosenbrock2D-200000-1-Eucl-GD.csv};
            \addplot[color=TolMutedSand, thick]
                table [x=i, y=y, col sep=comma] {data/Rosenbrock2D-200000-1-Eucl-DoC.csv};
            \addplot[color=TolMutedTeal, thick]
                table [x=i, y=y, col sep=comma] {data/Rosenbrock2D-200000-1-Riemannian-GD.csv};
            \addplot[color=TolMutedIndigo, thick]
                table [x=i, y=y, col sep=comma] {data/Rosenbrock2D-200000-1-Riemannian-DoC.csv};

            \legend{Euclidean GD, Euclidean DCA, Riemannian GD, Riemannian DCA}
        \end{axis}
    \end{tikzpicture}
    \caption{A comparison of the cost function during the iterations of the four experiments performed on the Rosenbrock example.}\label{fig:Rosenbrock}
\end{figure}
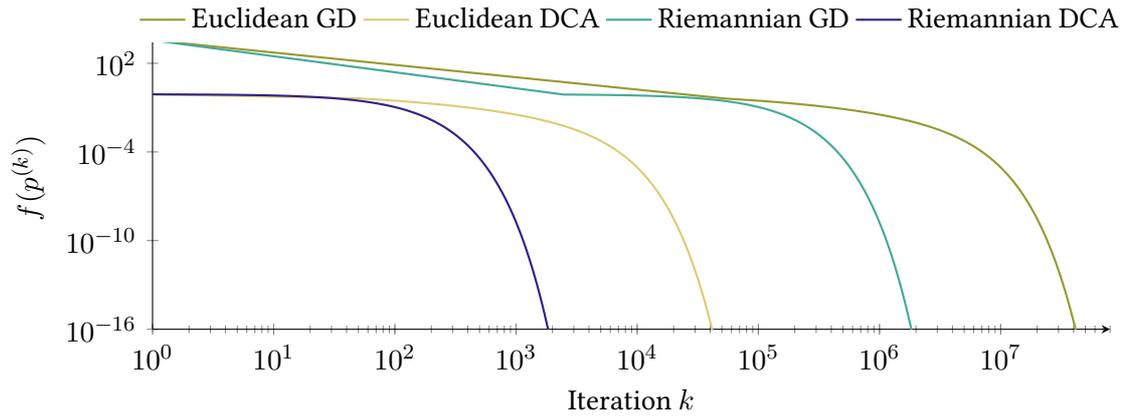
We set $b=1$ and $a=2\cdot10^5$. All algorithms start in $p^{(0)} = \frac{1}{10}(1, 2)^\mathrm{T}$.
The initial cost is $f(p^{(0)}) \approx 7220.81$.
The runtime and number of iterations is depicted in \cref{tab:Rosenbrock}
and the development of the cost function during the iterations in \cref{fig:Rosenbrock}.

For the cost $f(p^{(k)})$ during the iterations, we can observe that both gradient
algorithms as well as both difference of convex algorithms perform similar in shape,
both groups even have similar gain in their first step.

Still, even for the Euclidean case, the gradient descent with Armijo step size
requires several orders of magnitude more iterations than the Euclidean difference
of convex algorithm. The Riemannian gradient descent outperforms the Euclidean one
both in number of iterations as well as overall runtime.
Since a single iteration in the difference of convex algorithm requires to solve
a sub optimization problem, and we even employ a gradient descent per iteration,
even the Euclidean DCA is slower than the Riemannian gradient descent, while
the DCA already requires about a factor of 50 less iterations.
Similarly, the Riemannian DCA requires a factor of 1\,000 less iterations than the
Riemannian gradient descent, but since a single iteration is more costly, it
is only about a factor of 2 faster.
\begin{table}[tbp]
    \centering
    \begin{tabular}{lrr}
        \toprule
        \textbf{Algorithm} & \textbf{Runtime} & \textbf{\# Iterations} \\
        \midrule
        Euclidean GD & 305.567\,sec.& 53\,073\,227 \\
        Euclidean DCA & 58.268\,sec.& 50\,588\\
        Riemannian GD & 18.894\,sec.& 2\,454\,017 \\
        Riemannian DCA & 7.704\,sec.& 2\,459 \\
        \bottomrule
    \end{tabular}
    \caption{Summary of the runtime and number of iterations of the four experiments performed on the Rosenbrock example.}\label{tab:Rosenbrock}
\end{table}
\subsection{Constrained maximization of the Fréchet variance}
\label{sec:logdet}

Let $\mathcal M$ be the  manifold \eqref{eq:RiemPmatrix} of symmetric positive definite matrices $\mathbb P^n_{++}$, $n\in\mathbb N$
with the affine invariant metric $\langle\cdot,\cdot\rangle$,
$\{q_1, \dots, q_m\}\subset {\mathcal M}$ be a data set of distinct points, \ie $q_i\neq q_j$ for $i\neq j$,
and  $\mu_1, \ldots \mu_m$ be non-negative weights with $\sum_{j=1}^m\mu_j=1$.
Let $h\colon \mathcal M\to \mathbb R$ be the function defined by
\begin{equation}
    \label{eq:ff}
    h(p) \coloneqq
        \sum_{j=1}^m \mu_j d^2(p,q_i),
        \qquad \text{where }
        d^2(p,q_i)\coloneqq\operatorname{tr}\bigl(
            \log^2(p^{-1/2}q_jp^{-1/2})
        \big).
\end{equation}
Recall that
\begin{equation*}
    d(p, q) \coloneqq
    \lVert \log(p^{-1/2}qp^{-1/2})\rVert_{F}
    = \sqrt{\operatorname{tr}\log^2(p^{-1/2}qp^{-1/2})}
\end{equation*}
is the Riemannian distance between $p$ and $q$ on $\mathcal M$.
When every one of the weights  ${\mu}_1, \ldots {\mu}_m$ are equal,
this function $h$ is known as the \emph{Fréchet variance}
of the set $\{q_1, \dots, q_m\}$, see~\cite{Horev2017}.
In this example we want to consider the constrained Fréchet variance maximization problem, which is stated as
\begin{equation}\label{eq:probmaxpddK}
     \operatorname*{arg\,max}_{p\in\mathcal C} h(p)
\end{equation}
where the constrained  convex  set is  given by
\begin{equation}\label{eq:boxsetKd}
    {\mathcal C} \coloneqq
        \{ p\in {\mathcal M}\ |\ \bar q_{-}\preceq p \preceq \bar q_{+} \},
\end{equation}
where $\bar q_-,  \bar q_+\in\mathcal M$ with $\bar q_{-}\prec  \bar q_{+}$. Here, $p \prec q$ ($p \preceq q$) denotes the (non-strict) partial ordering on the spd-matrices, i.\,e.\ that $q-p$ is positive (semi-)definite or for short $q-p \prec 0$ ($\preceq 0$). We point out that \cite[Lemma 2 (iii)]{Lim2012} implies that the set $C$ is convex.

The problem~\eqref{eq:probmaxpddK} can be equivalently stated
as a Difference of Convex problem or a non-convex minimization problem.
The second formulation can be algorithmically solved by a Frank-Wolfe
algorithm~\cite{WeberSra2022}.

\paragraph{Maximizing the Fréchet variance as a DC problem.}

We define the \emph{indicator function} of the set $\mathcal C$ as
\begin{equation*}
    \iota_{\mathcal C}(p) = \begin{cases}
        0 & \text{ if } p \in \mathcal C\\
        \infty & \text{ else.}
    \end{cases}
\end{equation*}
Using $g = \iota_{\mathcal C}$ and the fact that
\begin{equation*}
    \operatorname*{arg\,max}_{p\in\mathcal C} h(p)
    =
    \operatorname*{arg\,min}_{p\in\mathcal C} -h(p)
\end{equation*}
to rephrase problem~\eqref{eq:probmaxpddK} to
\begin{equation}
    \label{eq:probmaxef}
    -\operatorname*{arg\,min}_{p\in\mathcal M} f(p),
    \qquad\text{where }  f(p)\coloneqq g(p)-h(p).
\end{equation}
We obtain indeed for $p\in \mathcal C$ that $f(p) = -h(p)$ and hence at a minimizer of $f$ we obtain a maximizer of $h$.
This hence yields a DC problem as studied in the previous sections.

By using~\eqref{eq:Grad}  and~\eqref{eq:ff}, the gradient $\grad h(p)$ is given by
\begin{align}  \label{ed:subgrdpdm1dp}
    \grad h(p)
    &= -2\sum_{j=1}^{m}\mu_jp^{1/2}\log({p^{-1/2}q_jp^{-1/2}}) p^{1/2}\\
    &= 2\sum_{j=1}^{m}{\mu_j}p^{1/2} \log({p^{1/2}q_j^{-1}p^{1/2}}) p^{1/2}.
    \label{ed:subgrdpdm2dp}
\end{align}
In this case, due to $g(p)=0$, for $p\in\mathcal C$,  the subproblem~\eqref{eq:DCAS}  for $X^{(k)}=\grad h(p^{(k)})$ is given by
\begin{equation}
    \label{eq:DCASEc}
    p^{(k+1)}\in \argmin_{p\in  {\mathcal C}}
        \big\langle- \grad h(p^{(k)}), \exp^{-1}_{p^{(k)}}p\big\rangle.
\end{equation}
On the other hand, it follows from~\eqref{eq:metric} and~\eqref{eq:invexp} that
\begin{align}
 \big\langle-\grad h(p^{(k)}),& \exp^{-1}_{p^{(k)}}p\big\rangle
 \\
    & = \big\langle
        -\grad h(p^{(k)}),
        {\bigl(p^{(k)}\bigr)}^{1/2} \log
        \Bigl(
            {{\bigl(p^{(k)}\bigr)}^{-1/2}p{\bigl(p^{(k)}\bigr)}^{-1/2}}\Bigr) {\bigl(p^{(k)}\bigr)}^{1/2}
    \big\rangle\\
    & = \operatorname{tr}\Bigl(
        -{\bigl(p^{(k)}\bigr)}^{-1/2} \grad h(p^{(k)}){\bigl(p^{(k)}\bigr)}^{-1/2}\log({{\bigl(p^{(k)}\bigr)}^{-1/2}p{\bigl(p^{(k)}\bigr)}^{-1/2}})
    \Bigr).
\end{align}
Therefore, the problem~\eqref{eq:DCASEc} becomes
\begin{equation} \label{eq:DCASEcef}
                p^{(k+1)}\in \argmin_{p\in  {\mathcal C}} \operatorname{tr}\big(s^{(k)} \log({{\bigl(p^{(k)}\bigr)}^{-1/2}p{\bigl(p^{(k)}\bigr)}^{-1/2}})\big),
\end{equation}
where, by using~\eqref{ed:subgrdpdm1dp}, the matrix $s^{(k)}$ is given by
\begin{equation}  \label{ed:subgrdpdm1it}
s^{(k)} = -{\bigl(p^{(k)}\bigr)}^{-1/2} \grad h(p^{(k)}){\bigl(p^{(k)}\bigr)}^{-1/2}= 2\sum_{j=1}^{m}{\mu_j}\log({\bigl(p^{(k)}\bigr)}^{-1/2}q_j{\bigl(p^{(k)}\bigr)}^{-1/2}).
\end{equation}
or, by using~\eqref{ed:subgrdpdm2dp},  the  matrix $s^{(k)}$ is given  equivalently  by
\begin{equation}  \label{ed:subgrdpdm2it}
s^{(k)}\coloneqq-{\bigl(p^{(k)}\bigr)}^{-1/2} \grad h(p^{(k)}){\bigl(p^{(k)}\bigr)}^{-1/2}= -2\sum_{j=1}^{m}{\mu_j}\log({\bigl(p^{(k)}\bigr)}^{1/2}q_j^{-1}{\bigl(p^{(k)}\bigr)}^{1/2}).
\end{equation}
 In order to deal with  the subproblem \eqref{eq:DCASEcef} we consider the following theorem,  which gives a closed formula for it , see  \cite[Theorem~4]{WeberSra2022}.
\begin{theorem}\label{th:esmw}
Let $L,U\in {\mathbb P}^{n}_{++}$ such that $L\prec U$.
Let $S$ be a Hermitian $(n\times n)$ matrix and $X\in{\mathbb P}^{n}_{++}$ be arbitrary.
Then, the solution to the optimization problem
\begin{align*}
\min_{L\preceq Z\preceq U} \mbox{\emph{tr}}(S\log (XZX)),
\end{align*}
is given by $Z=X^{-1}Q\left( P^{\top}[-\mbox{sgn}(D)]_{+}P+\hat{L} \right)Q^{\top}X^{-1}$, where $S=QDQ^{\top}$ is a diagonalization of $S$, $\hat{U}-\hat{L}=P^{\top}P$ with $\hat{L}=Q^{\top}XLXQ$ and $\hat{U}=Q^{\top}XUXQ,$ where $[-\mbox{sgn}(D)]_{+}$ is the diagonal matrix
$$\mbox{\emph{diag}}\left([-\mbox{sgn}(d_{11})]_{+}, \ldots, [-\mbox{sgn}(d_{nn})]_{+} \right)$$ and $D=(d_{ij}).$
\end{theorem}
\begin{remark}
    The solution to \eqref{eq:DCASEcef}
    can be obtained \cref{th:esmw} setting $L=\bar q_-$, $U=\bar q_+$, $S = s^{(k)}$, $X=\bigl(p^{(k)}\bigr)^{-\frac{1}{2}}$ and $Z=p$.
    Note that given $p^{(k)}$ both $X$ and $X^{-1}$ can be easily computed using the eigen decomposition and modifying the diagonal matrix.
\end{remark}
To minimize a constrained, non-convex function $f_{\mathrm{FW}} \colon \mathcal X \to \mathbb R$, $\mathcal X \subset \mathcal M$, \cite{WeberSra2022} propose the Riemannian Frank-Wolfe algorithm as summarized in \cref{Alg:FW}.

\begin{algorithm}[hbp]
    \caption{The Riemannian Frank-Wolfe Algorithm, cf.~\cite[Algorithm 2]{WeberSra2022}.}\label{Alg:FW}
    \begin{algorithmic}[1]
        \STATE {
        Choose an initial point $p^{(0)}\in \mathcal X$.
        Set $k=0$.
        }
        \WHILE{convergence criterion is not met}
            \STATE $q^{(k)} \gets
                \displaystyle\operatorname*{argmin}_{q\in\mathcal X}\ %
                    \bigl\langle
                    \operatorname{grad} f_{\mathrm{FW}}(p^{(k)}),
                    \exp_{p^{(k)}}^{-1}q
                \bigr\rangle$
        \STATE $s_k \gets \frac{2}{2+k}$
        \STATE $p^{(k+1)} = \gamma_{p^{(k)}q^{(k)}}(s_k)$
        \STATE $k \gets k+1$
        \ENDWHILE
    \end{algorithmic}
\end{algorithm}

In our example we have $\mathcal X = \mathcal C$ and $f_{\mathrm{FW}} = -h$, \ie  we obtain a concave constrained problem. Since $f_{\mathrm{FW}} = -h$, we obtain the same subproblem in Step 3 of the Frank-Wolfe Algorithm as stated in~\eqref{eq:DCASEc}.
Thus, we have two algorithms for solving the problem~\eqref{eq:probmaxpddK} or equivalently~\eqref{eq:probmaxef},
namely \cref{Alg:FW} and \cref{Alg:DCA2}.
Both possess the same subproblem in this case.
They treat the result of the subproblem differently, though.
While \cref{Alg:DCA2} uses the subproblem solution directly for the next iteration, \cref{Alg:FW} uses the solution as an end point of a geodesic segment  starting from the previous iterate. This geodesic segment  is then evaluated at a certain interims point.
This also means that \cref{Alg:FW} has to start in a \emph{feasible point} $p^{(0)} \in \mathcal C$, while for \cref{Alg:DCA2} this is not necessary.
\\
In our numerical example we consider $\mathcal M = \mathbb P_{++}^{20}$, that is, the set of $20\times20$ symmetric positive definite matrices with the affine-invariant metric $\langle\cdot,\cdot\rangle$. This is a Riemannian manifold of dimension $d=210$. We further generate $m=100$ random spd\ matrices $q_i$ with corresponding random weights $w_i$ as the data set for the Fréchet variance.
We set
\begin{equation*}
    q_{-} \coloneqq \biggl(\sum_{i=1}^m w_iq_i^{-1}\biggr)^{-1},
    \qquad
    q_{+} \coloneqq \sum_{i=1}^m w_iq_i,
    \quad\text{ and }\quad
    p^{(0)}\coloneqq \frac{1}{2}(q_{-}+q_{+}).
\end{equation*}
A numerical implementation of \cref{th:esmw} is used as a closed-form solver of the subproblem.
Numerically we observe, that these results might suffer from
imprecisions, which means they might not meet the constraint, but only by around $2\cdot10^{-13}$.
Since we use these points as iterates in \cref{Alg:DCA2}, only for this algorithm we add a “safeguard” and perform a small line search for the first matrix closest matrix to the sub-solver's result $q^*$ on the geodesic to the last iterate $p^{(k)}$ fulfilling the constraint.

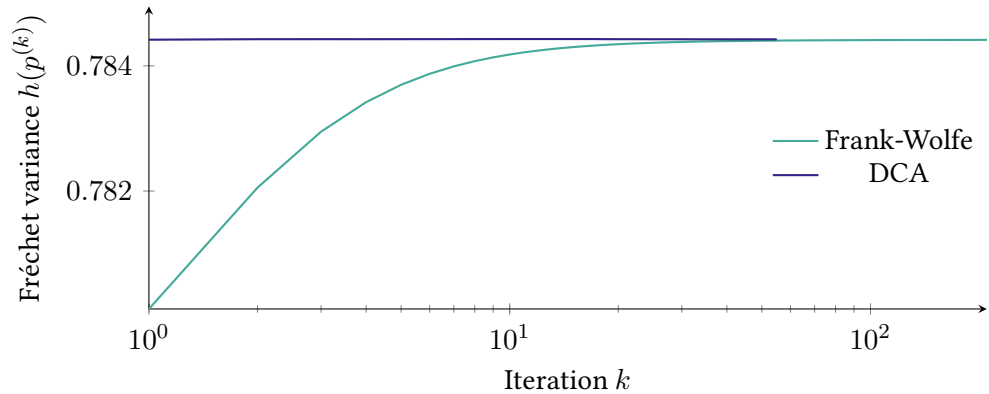
\begin{figure}[tbp]\centering
    \begin{tikzpicture}
        \begin{axis}[
            xmode=log,
            y tick label style={
              /pgf/number format/.cd,
                fixed,
                fixed zerofill,
                precision=3,
                /tikz/.cd
            },
            ymax=0.78495,
            xmin=1,
            xmax=210,
            legend style={draw=none, at={(1.0,0.5)},anchor=east},
            width=.8\textwidth,
            height=.25\textheight,
            axis lines=left,
            ylabel near ticks,
            xlabel near ticks,
            xlabel={Iteration $k$},
            ylabel={Fréchet variance $h(p^{(k)})$},
        ]
        \addplot[color=TolMutedTeal, thick]
        table [x=i, y=h_fw, col sep=comma] {data/DCAvsFrankWolfe-n20-m100-experiment1_fw_short.csv};
        \addplot[color=TolMutedIndigo, thick]
                table [x=i, y=h_dc, col sep=comma] {data/DCAvsFrankWolfe-n20-m100-experiment1_dc.csv};
            \legend{Frank-Wolfe, DCA}
        \end{axis}
    \end{tikzpicture}
    \caption{A comparison of the Fréchet variance $h(p)$ during the iterations of \cref{Alg:DCA2} (indigo) and \cref{Alg:FW} (teal).
    }\label{fig:DCAvsFW}
\end{figure}

Then we stop \cref{Alg:DCA2} if the change $d(p^{(k)}, p^{(k+1)}) < 10^{-14}$ or if the gradient change between these two iterates (computes using parallel transport) is below $10^{-9}$.
The algorithm stops after $55$ iterations due to a small gradient change.
\\
For Frank-Wolfe a suitable stopping criterion is challenging.
Note that the gradient $\operatorname{grad} f_{\mathrm{fw}}$ does not tend to $0$ of the minimizer is on the boundary. Even after $100\,000$ iterations, Frank-Wolfe still has not reached either of the stopping criteria, both changes are still of order $10^{-4}$.
While the Difference of Convex Algorithm reaches it's minimum (its maximum in $h$) after $11$ iterations and then increases slightly, probably due to the closed-form solution not being precise, Frank-Wolfe reaches neither of these two values – after $11$ or $55$ in these $100\,000$ iterations.
\\
Finally, comparing the time per iteration, both algorithms comparable. With the numerical safeguard for this specific problem, $1000$ iterations of DCA take $16.01$ seconds,
Frank-Wolfe $8.13$ seconds and DCA without the safeguard $7.475$ seconds. That is, in runtime per iteration, using the same sub solver, both perform similarly, while DCA seems to have a vanishing gradient change.
\section{Conclusion}\label{sec:Conclusion}
In this paper, we investigated the extension of the Difference of Convex Algorithm (DCA) to the Riemannian case, enabling us to solver DC problems on Riemannian manifolds. We investigate its relation to Duality on manifolds and state a convergence result on Hadamard manifolds.

Numerically, the new algorithm outperforms the existing Difference of Convex Proximal Point algorithm (DCPPA) in terms of the number of iterations. However, for large-dimensional manifolds, the DCPPA is faster. Additionally,  for a specific class of constrained maximization problems, the DCA is well-suited and outperforms the Riemannian Frank-Wolfe algorithm, especially because a suitable stopping criterion can be used but also in how close it gets to the actual minimizer.
Finally, rephrasing Euclidean problems into DC problems with a suitable metric is another field where using the DCA seems very beneficial.

Extending the numerical algorithms also to employing duality is a future research topic, where the iteration time and convergence speed might increase.

\printbibliography
\end{document}